\documentclass[12pt]{amsart}

\usepackage[english]{babel}
\usepackage[toc,page]{appendix}
\usepackage{amsmath}
\usepackage{amsthm}
\usepackage{amssymb}
\usepackage{mathrsfs}
\usepackage{enumerate}
\usepackage{mathscinet}
\usepackage[notcite, final, notref]{showkeys}
\usepackage{dsfont}
\usepackage{bookmark}
\usepackage{hhline}
\usepackage[dvips]{color}
\usepackage{tikz}
\usetikzlibrary{decorations.markings}
\tikzset{>=latex}

\newcommand{\suchthat}{\;\ifnum\currentgrouptype=16 \middle\fi|\;}

\newtheorem{theorem}{Theorem}[section]
\newtheorem{problem}[theorem]{Problem}
\newtheorem{lemma}[theorem]{Lemma}
\newtheorem{proposition}[theorem]{Proposition}
\newtheorem{corollary}[theorem]{Corollary}
\newtheorem{definition}[theorem]{Definition}

\theoremstyle{definition}
\newtheorem{remark}[theorem]{Remark}

\allowdisplaybreaks

\usepackage{tikz}
\usepackage[siunitx]{circuitikz}

\usepackage{xcolor}

\title[Positivity and some of its implications in the Grassmann algebra]{Positivity, rational Schur functions, Blaschke factors, and other related results in the Grassmann algebra}

\author[D. Alpay]{Daniel Alpay}
\address{(DA) 
Faculty of Mathematics, Physics, and Computation\\
Schmid College of Science and Technology\\
Chapman University\\
One University Drive\\
Orange, California 92866\\
USA}
\email{alpay@chapman.edu}

\author[I. L. Paiva]{Ismael L. Paiva }
\address{(ILP)
Schmid College of Science and Technology\\
Chapman University\\
One University Drive\\
Orange, California 92866\\
USA}
\email{depaiva@chapman.edu}

\author[D. C. Struppa]{Daniele C. Struppa}
\address{(DCS) 
Faculty of Mathematics, Physics, and Computation\\
Schmid College of Science and Technology\\
Chapman University\\
One University Drive\\
Orange, California 92866\\
USA}
\email{struppa@chapman.edu}

\begin{document}

\tikzset{->-/.style={decoration={
  markings,
  mark=at position #1 with {\arrow[scale=2]{>}}},postaction={decorate}}}

\begin{abstract}
We begin a study of Schur analysis in the setting of the Grassmann algebra, when the latter is completed with respect to the $1$-norm. We focus on the rational case.
We start with a theorem on 
invertibility in the completed algebra, and define a notion of positivity in this setting. We present a series of applications pertaining to Schur analysis, including 
a counterpart of the Schur algorithm, extension of matrices and rational functions. Other topics considered include Wiener algebra, reproducing kernels Banach
modules, and Blaschke factors. 
\end{abstract}

\maketitle



\date{today}
\setcounter{tocdepth}{1}
\tableofcontents


\section{Introduction}
\setcounter{equation}{0}
\label{intro}
The purpose of this work is to begin a study of Schur analysis and related topics in the setting of the Grassmann algebra. To put the problem into perspective and 
to set the framework, we first briefly review the classical setting. Schur analysis is part of function theory in the open unit disk $\mathbb{D}$ or in a half-plane. 
It is a rich and vastly developed field with numerous applications, which include -- but are not limited to -- signal processing \cite{MR1677962}, fast algorithms \cite{C-schur} and linear systems \cite{MR3362408}. It originated with the work of Schur \cite{schur,schur2}, although this area can be even traced back to Stieltjes \cite{MR1623484}. We suggest the reference \cite{hspnw} for a collection of related original papers on the topic. The Hardy space and Blaschke factors are important players in this domain (see e.g. \cite{MR1102893, sarason94}), as well as the Wiener algebra and rational functions.\smallskip

Recall that the Hardy space of the unit disk $\mathbf H_2$ is the Hilbert space of power series $f(\lambda)=\sum_{n=0}^\infty a_n\lambda^n$, $\lambda\in\mathbb{C}$, such that $\|f\|_{\mathbf H_2}^2\equiv\sum_{n=0}^\infty |a_n|^2<\infty$. From the signal processing point of view, it can be interpreted as the space of $z$-transforms of finite energy discrete signals. Such an interpretation motivates various interpolation problems in the Hardy space and related spaces. The Nevanlinna-Pick and the Carath\'{e}odory-Fej\'{e}r problems are two examples of it. The former consists on, given $\lambda_1,\ldots, \lambda_N$ in $\mathbb D$ and complex numbers $\omega_1,\ldots, \omega_N$, describing the set of all Hardy functions $f$ such that $f(\lambda_j)=\omega_j$ for $j=1,\ldots N$. The latter, on the other hand, refers to the problem of fixing the first $N$ derivatives of a function at a given point. In both cases, additional metric constraints are made on $f$, such as taking contractive values in the open unit disk. Such functions are called Schur functions, and are transfer functions of dissipative systems. See, for instance, \cite{MR2002b:47144,Dym_CBMS,MR84e:93003,Fuhrmann,helton,helton78}.\smallskip

In \cite{schur}, using Schwarz' lemma, Schur associated to a Schur function $s(\lambda)$ a sequence, finite or infinite, of Schur functions $s_0,s_1\ldots$ via the recursion
\begin{equation}
\label{recurschur}
\begin{split}
s_0(\lambda)&=s(\lambda)\\
s_{n+1}(\lambda)&=\frac{s_n(\lambda)-s_n(0)}{\lambda(1-s_n(\lambda)\overline{s_n(0)})},\quad n=0,1,\ldots\\
\end{split}
\end{equation}
Such a recursion ends at a rank $n$ if $|s_n(0)|=1$, and this happens if and only if $s$ is a finite Blaschke product. The numbers $\rho_n=s_n(0)$, $n=0,1,\ldots$ are called the Schur coefficients of $s$; they lead to a continued fraction expansion of $s$, and prove more appropriate than the Taylor series of $s$ to solve various approximation problems -- see, e.g., \cite{C-schur,Dym_CBMS}.\smallskip

For $\lambda=0$, the Carath\'{e}odory-Fej\'{e}r problem becomes trivial for Hardy functions -- since the coefficients $a_n$ of $f$ are known. However, this is a problem of central importance in the class of functions with positive real part in the open unit disk, which is related to the theory of extension of Toeplitz matrices and has applications on the prediction theory of second-order stationary processes. \smallskip

Boundary values of Hardy functions can be quite a challenging problem. It is, then, sometimes desirable to consider functions in the Wiener algebra $\mathcal W_+$. To define the latter, we first introduce the Wiener algebra $\mathcal{W}^p$:
\[
\mathcal{W}^p = \left\{f(e^{it}) = \sum_{n\in\mathbb{Z}}e^{int}f_n \suchthat t\in\mathbb{R}, f_n\in\mathbb{C}^{p\times p}, \sum_{n\in\mathbb{Z}}\Vert f_n\Vert<\infty\right\},
\]
where $\Vert\cdot\Vert$ is a complex matrix norm, endowed with the usual multiplication of functions. Some facts about $\mathcal{W}^p$ are relevant for our present discussion. For instance, a function $f\in\mathcal{W}^p$ is said to be strictly positive if 
$f(e^{it})>0$ for every real $t$. Moreover, the Wiener-L\'{e}vy theorem assures that $f$ has an inverse in $\mathcal{W}^p$ if and only if $f(e^{it})\neq 0$ for every 
real $t$, i.e., invertibility in the algebra is equivalent to pointwise 
invertibility.\smallskip

Two important subalgebras of $\mathcal{W}^p$ are $\mathcal{W}^p_+$, composed by functions $f$ of the type $f(e^{it})=\sum_{n=0}^\infty e^{int}f_n$, and $\mathcal{W}^p_-$, which contains the functions $f(e^{it})=\sum_{n=-\infty}^0 e^{int}f_n$. There are elements in $\mathcal{W}^p_+$ (resp. $\mathcal{W}^p_-$) that have inverse in $\mathcal{W}^p_+$ (resp. $\mathcal{W}^p_-$). They are denoted by $f_+$ (resp. $f_-$). Furthermore, some functions in $\mathcal{W}^p$ can be factorized as $f=f_+f_-$. A theorem states that $f\in\mathcal{W}^p$ is strictly positive if and only if it has such a factorization and it is characterized by $f_-=f_+^*$, where $f_+^*$ denotes the adjoint of $f_+$. \smallskip

Extending these notions to more general settings -- to name a few, several complex variables, upper triangular operators, quaternionic analysis, bi-complex numbers -- 
has been a source of new problems and methods. See, e.g., \cite{MR93b:47027,MR3309382,a-volok,dede,MR99g:93001}. Each of those settings has a natural interpretation in terms of signal theory and linear systems. For example, time-varying systems correspond to upper triangular operators, and systems indexed by several indices correspond to function theory in the unit ball of $\mathbb{C}^N$ or the unit polydisk. Moreover, in all those settings there exists a natural counterpart of the Hardy space. In the case of the upper triangular operators, it is the space of Hilbert Schmidt upper triangular operators. As to the case of function theory in the unit ball, on the other hand, it is the Drury-Arveson space \cite{MR80c:47010}, which is different from the classical Hardy space when the dimension $N$ is greater than $1$. In the case of the work we present here, the counterpart of the Hardy space is a Wiener-type algebra.\smallskip

The central aspect of our approach is the replacement of the complex numbers by the Grassmann algebra $\Lambda$. The latter plays a fundamental role in supersymmetry and, also, in quantum field theory, where it allows the construction of path integrals for fermions \cite{das1993field}. We recall that $\Lambda$ is the unital algebra on the complex numbers generated by $1$ and a countable set of elements $i_n$ not belonging to $\mathbb{C}$, linearly independent over $\mathbb{C}$, and satisfying
\begin{equation}
i_ni_m+i_mi_n=0, \label{anticommutative}
\end{equation}
where $n,m=1,2,\ldots$, and in particular
\begin{equation}
i_n^2 = 0.
\end{equation}
An element of $\Lambda$ is often referred to as a supernumber. If we denote by $\Lambda_N$ the case with $N$ generators $i_n$, we write $\Lambda=\cup_{N\in\mathbb{N}}\Lambda_N$. This differs from the usual way $\Lambda$ is treated in the literature, since if $z\in\Lambda$, there exists $n(z)$ such that $z\in\Lambda_{n(z)}$. Here, we follow the approach introduced in \cite{rogers1980global,zbMATH00861741}to study problems with an effective infinite number of generators, we look at a closure of $\Lambda$ with respect to a norm that makes the closure a Banach 
algebra -- see expression \eqref{1norm}.\smallskip

The paper consists of eleven sections besides this introduction. We now 
describe their content with an emphasis on the main results. In Section \ref{sec2}, we consider the closure $\overline{\Lambda}^{(1)}$ of the Grassmann algebra under the $1$-norm \eqref{1norm}. In particular, we define positive numbers and, using Gelfand theory, we characterize invertibility in this non-commutative Banach algebra. Building on these, we study in Section \ref{sec3} matrices with entries in $\overline{\Lambda}^{(1)}$. We introduce, in particular, the notion of positive matrix and prove a factorization theorem for them. As an example of application of these results, we study in Section \ref{sec-matrix-ext} the one step extension problem for Toeplitz matrices with entries in $\overline{\Lambda}^{(1)}$. In Section \ref{ratio}, we define and characterize rational functions, and  introduce in particular the notion of a realization in this setting. To have such notions, we need a special product for power series with coefficients in $\overline{\Lambda}^{(1)}$, namely the convolution on the coefficients (or the Cauchy product), which is denoted here by $\star$. Built these tools, we are ready to introduce and characterize rational Schur functions in Section \ref{sec-schur}. Moreover, using the Cauchy product, we introduce and study in Section \ref{sec-wiener} the Grassmannian counterpart of the Wiener algebra. In Section \ref{sec-rkpm}, we study finite dimensional reproducing kernel modules over $\overline{\Lambda}^{(1)}$ and characterize such modules associated to $J$-unitary rational functions in the present setting -- see Definition \ref{rationalpotapov} for the latter. We also consider an interpolation problem in the setting of the Wiener algebra. Schur analysis itself is studied in the next three sections: In Section \ref{sec-interpolation}, we study the Nevanlinna-Pick interpolation and, in Section \ref{schuralgo}, we introduce the Schur algorithm. Blaschke and Brune factors, which in the classical case may be considered the building blocks of Schur analysis are presented in Section \ref{sec-blaschke}. Finally, we lay out in Section \ref{sec-final} future directions of research.

\section{Some aspects of $\Lambda$}
\setcounter{equation}{0}
\label{sec2}

In this section, we introduce the idea of positivity in the Grassmann algebra $\Lambda$ and also some results on its closure with respect to the $1$-norm -- see Definition \ref{1norm}. Before that, we review some basic definitions on this setting.

\begin{definition}
We denote by $\mathfrak I$ the set of $t$-uples $({a_1},\ldots, {a_t})\in\mathbb N^t$, where $t$ runs through $\mathbb N$ and $a_1<{a_2}<\cdots<{a_t}$. For $\alpha=({a_1},\ldots, {a_t})\in\mathfrak I$ we set $i_\alpha=i_{a_1}\cdots i_{a_t}$ and write an element $z\in\Lambda$ as a finite sum
\begin{equation}
\label{Grassmann-decomposition}
z=z_0+\sum_{\alpha\in\mathfrak{I}}z_\alpha i_\alpha,
\end{equation}
where the coefficients $z_0$ and $z_{a_1,\ldots,a_k}$ are complex numbers.
\end{definition}

The term that does not contain any Grassmann generator, $z_0$, is called the {\it body} of the number and is often denoted by $z_B$, while $z_S=z-z_B$ is said to be the {\it soul} of the number \cite{MR1172996}. One can also give a meaning to the sum \eqref{Grassmann-decomposition} when it has an infinite number of terms, as we do later in this section.\smallskip

We also set $i_0\equiv1$ and ``extend'' the set $\mathfrak{I}$ by defining $\mathfrak{I}_0\equiv\{0\}\cup\mathfrak{I}$. Hence, a supernumber can be simply written as
\[
z = \sum_{\alpha\in\mathfrak{I}_0} z_\alpha i_\alpha.
\]

If $z=\sum_{\alpha\in\mathfrak{I}_0}z_\alpha i_\alpha$ and $w=\sum_{\beta\in\mathfrak{I}_0}w_\beta i_\beta$, their product makes sense since the sums are finite, and can be written as
\[
zw = \sum_{\alpha,\beta\in\mathfrak{I}_0}z_\alpha w_\beta i_\alpha i_\beta.
\]
Let $\alpha,\beta\in\mathfrak{I}$ and note that $i_\alpha i_\beta=0$ when $i_\alpha$ and $i_\beta$ have a common factor $i_u$, with $u\in\mathbb N$. Moreover, when $i_\alpha i_\beta$ does not vanish, it might still not be an element of the set $\{i_\alpha : \alpha\in\mathfrak{I}\}$, since permutations might be necessary to obtain such type of element. However, because permutations only introduce powers of negative one, there exists a uniquely defined $\gamma\in\mathfrak{I}$ such that
\[
i_\alpha i_\beta = (-1)^{\sigma(\alpha,\beta)} i_\gamma,
\]
where $\sigma(\alpha,\beta)$ is the number of permutations necessary to ``build'' $\gamma$ from $\alpha$ and $\beta$. If such a relation holds, we write
\begin{equation}
\alpha\vee\beta=\gamma.
\label{vee}
\end{equation}
So $i_\alpha i_\beta = (-1)^{\sigma(\alpha,\beta)} i_{\alpha\vee\beta}$. \smallskip

\begin{proposition}
A supernumber in $\Lambda$ is invertible if and only if its body is different from zero.
\label{invertibility-lambda}
\end{proposition}

\begin{proof}
Let $z=z_B+z_S\in\Lambda$. As a consequence, there exists $n(z)$ such that $z\in\Lambda_{n(z)}$. In particular, it is easy to check that $z_S^{n(z)+1}=0$. Therefore, if $z_B\neq0$ the expression
\[
z^{-1} = z_B^{-1} \sum_{k=0}^{n(z)} \left(-\frac{z_S}{z_B}\right)^k
\]
gives the inverse of $z$.

Conversely, assume $z$ is invertible and let its inverse be $w=w_B+w_S\in\Lambda$. Then,
\[
zw=1\Rightarrow z_Bw_B=1\Rightarrow z_B\neq 0.
\]
\end{proof}

The next results concern the square root, denoted by $\sqrt{z}$ or by $z^{1/2}$, of a supernumber $z\in\Lambda$. By square root of a supernumber, we mean the analytic extension of the usual square root of a complex number.

\begin{proposition}
Every invertible supernumber in $\Lambda$ has a square root.
\label{prop-sqrt-lambda}
\end{proposition}

\begin{proof}
Let $z=z_B+z_S\in\Lambda$ be an invertible supernumber. Because of Proposition \ref{invertibility-lambda}, we can assume $z_B\neq0$. Therefore,
\[
z=z_B\left(1+\frac{z_S}{z_B}\right)\Rightarrow \sqrt{z}=\sqrt{z_B}\sqrt{1+\frac{z_S}{z_B}}=\sqrt{z_B}\left[1-\sum_{k=0}^\infty\frac{2}{k+1}\binom{2k}{k}\left(-\frac{z_S}{4z_B}\right)^{k+1}\right].
\]
Observe that the last sum converges because it is finite. In fact, because $z\in\Lambda$, there exists $n(z)$ such that $z\in\Lambda_{n(z)}$, which implies that $z_S^{n(z)+1}=0$.
\end{proof}

It is common to define the conjugated $\dag$ of a supernumber $z$ as
\[
z^{\dag} \equiv \overline{z_0}+\sum_{\alpha\in\mathfrak I} (-1)^{\pi(\alpha)} \overline{z_\alpha} i_\alpha,
\]
where $\pi(\alpha) = |\alpha|(|\alpha|-1)/2$ with $|\alpha|$ being the number of elements of $\alpha$. Observe that it can be characterized as the complex conjugation of the coefficients $z_\alpha$, $i_n^{\dag} = i_n$, and $(zw)^{\dag} = w^{\dag} z^{\dag}$. In \cite{2018arXiv180611058A}, we showed how this conjugation can be understood as the composition of other conjugations that emerge, in some sense, more naturally from the symmetry of $\Lambda$.\smallskip

Following \cite{MR1172996}, if $z\in\Lambda$ is such that $z^\dag=z$, it is said to be a {\it real supernumber}, or superreal. On the other hand, if $z^\dag=-z$, it is said to be an {\it imaginary supernumber}. Note that a real supernumber generally does not belong to $\mathbb{R}$. For instance, $i_1+i_2$ is a real supernumber.\smallskip

Now, we add to those definitions the idea of a {\it non-negative} and a {\it non-positive} supernumber.

\begin{definition}
\label{posi_num}
Let $z\in\Lambda$. If there exists $w\in\Lambda$, such that $z=ww^\dag$ (resp. $z=-ww^\dag$), $z$ is said to be a non-negative (resp. non-positive) supernumber, and we write $z\succeq0$ (resp. $z=-ww^\dag$).
\end{definition}

\begin{proposition}
A non-negative or non-positive supernumber is, in particular, a real supernumber.
\label{positive-real}
\end{proposition}

\begin{proof}
This follows directly from Definition \ref{posi_num} since if $z=ww^\dag$ or $z=-ww^\dag$ for some $w\in\Lambda$, then $z=z^\dag$.
\end{proof}

If one requires invertibility of $z$,  then $w$ is invertible and $z$ is called a {\it positive} (resp. {\it negative}) supernumber, or superpositive (resp. supernegative), and we use the notation $z\succ 0$ (resp. $z\prec 0$).

\begin{proposition}
Let $z\in\Lambda$ be a real supernumber. Then, the following holds:
\begin{itemize}
\item $z\succ0 \Leftrightarrow z_B>0$;
\item $z\prec0 \Leftrightarrow z_B<0$;
\item $z\succeq0 \Leftrightarrow z_B\geq0$;
\item $z\preceq0 \Leftrightarrow z_B\leq0$.
\end{itemize}
\label{body-part}
\end{proposition}

\begin{proof}[Outline of the proof]
The restriction to the body of the inequalities for $z$ follows trivially. To prove the converse, observe that the square root of positive supernumbers built in the proof of Proposition \ref{prop-sqrt-lambda} is also a real supernumber.
\end{proof}

\begin{remark}
Observe that the notion of superpositivity induces a partial order in $\Lambda$.
\end{remark}

Another important characteristic of $\Lambda$ that often plays a role in the results we present here is the fact that it is a $\mathbb{Z}_2$-graded algebra, i.e., $\Lambda=\Lambda_{even}\oplus \Lambda_{odd}$. The subset $\Lambda_{even}$ is characterized by elements $u$ of the type
\begin{equation}
u=u_0+\sum_{\substack{\alpha\in\mathfrak{I} \\ |\alpha| \ {\text even}}}u_\alpha i_\alpha,
\label{even-number}
\end{equation}
where $|\alpha|$ is the number of elements of $\alpha$. Those supernumbers are called the {\it even} supernumbers. They clearly commute with {\it every} element of $\Lambda$ and, furthermore, form a commutative subalgebra of $\Lambda$. \smallskip

The elements of $\Lambda_{odd}$ can be written as
\begin{equation}
v=\sum_{\substack{\alpha\in\mathfrak{I} \\ |\alpha| \ {\text odd}}}v_\alpha i_\alpha.
\label{odd-number}
\end{equation}
They are called {\it odd} supernumbers. It is easy to see that $\Lambda_{odd}$ is not a subalgebra of $\Lambda$ since the product of two odd supernumbers is an even supernumber.\smallskip

In \cite{2018arXiv180611058A}, we introduced the $p$-norm of a supernumber, proved an inequality they satisfy, and developed some aspects of $\Lambda$ endowed with the $2$-norm. Here, our focus is on aspects of the closure of $\Lambda$ with respect to the $1$-norm, which we define now.

\begin{definition}
The $1$-norm of a supernumber $z\in\Lambda$ is defined as
\begin{equation}
\Vert z \Vert_1 = \sum_{\alpha\in\mathfrak{I}_0} |z_\alpha|,
\label{1norm}
\end{equation}
where $|\cdot|$ is the usual modulus of a complex number.
\end{definition}

We denote by $\overline{\Lambda}^{(1)}$ the closure of $\Lambda$ with respect to the $1$-norm. This set was first introduced by Rogers \cite{rogers1980global,zbMATH00861741} and is usually used as the Grassmannian setting with countably infinite generators -- see e.g. \cite{hoyos1982superfiber,jadczyk1981superspaces,vladimirov1984superanalysis}.\smallskip

The following proposition follows trivially from the definition and, therefore, has its proof omitted.

\begin{proposition}
$\overline{\Lambda}^{(1)}$ endowed with the $1$-norm is a complex normed space.
\end{proposition}

\begin{proposition}
Let $a\in\overline{\Lambda}^{(1)}$ be such that $i_\alpha a=0$ for every $\alpha\in\mathfrak{I}$. Then, $a=0$.
\end{proposition}

\begin{proof}
It is clear that the body of $a$ should be null. Let us consider, then, the general case of a supernumber in $\overline{\Lambda}^{(1)}$ with pure soul, i.e.,
\[
a=\sum_{\beta\in\mathfrak{I}}a_\beta i_\beta.
\]
We assume there exists $a$ such that $i_\alpha a=0$ for every $\alpha$. We rewrite it as $a=b+c$ by defining
\[
b=\sum_{m\in\mathbb{N}}a_m i_m
\]
and
\[
c=\sum_{p=2}^{\infty}\sum_{\substack{\beta\in\mathfrak{I}; \\ |\beta|=p}} a_\beta i_\beta.
\]
Let $|\alpha|=1$ and $\alpha=n$, where $n$ is a natural number. In this case
\[
i_n b = \sum_{\substack{m\in\mathbb{N}; \\ m\neq n}}a_m i_ni_m.
\]
Since $i_na=0$, the right-hand side of the above expression should be canceled by some terms of $i_nc$. However, this cannot be the case, because if we write
\[
i_nc=\sum_{\beta\in\mathfrak{I}}c_\beta i_\beta,
\]
then we clearly get that $c_\beta=0$ for $|\beta|\leq 2$. Therefore, $b=a_ni_n$. However, since $i_nb=0$ for every $n\in\mathbb{N}$, our final conclusion is that $b=0$. Finally, an inductive reasoning for $|\beta|\geq 2$ shows that $a=0$.
\end{proof}

\begin{proposition}
Let $z,w\in\overline{\Lambda}^{(1)}$. Then, the following inequality holds
\begin{equation}
\Vert zw \Vert_1 \leq \Vert z \Vert_1 \Vert w \Vert_1
\label{1norm-ineq}
\end{equation}
and hence $\overline{\Lambda}^{(1)}$ is a Banach algebra isomorphic to $\ell_1(\mathfrak{I},\mathbb{C})$, when the latter is endowed with its monoid structure.
\end{proposition}

\begin{proposition}
Let $\overline{\Lambda}^{(1)}_{odd}$ be the odd part of $\overline{\Lambda}^{(1)}$ and $v\in\overline{\Lambda}^{(1)}_{odd}$. Then, $v^2=0$.
\label{odd-square}
\end{proposition}

\begin{proof}
Let $v\in\overline{\Lambda}_{odd}^{(1)}$. Then,
\begin{eqnarray*}
v^2 & = & \frac{1}{2}\sum_{\substack{\alpha,\beta\in\mathfrak{I} \\ |\alpha|,|\beta| \ {\text odd}}} \left(v_\alpha v_\beta i_\alpha i_\beta + v_\beta v_\alpha i_\beta i_\alpha\right) \\
    & = & \frac{1}{2}\sum_{\substack{\alpha,\beta\in\mathfrak{I} \\ |\alpha|,|\beta| \ {\text odd}}}v_\alpha v_\beta \left(i_\alpha i_\beta + i_\beta i_\alpha\right) \\
    & = & 0,
\end{eqnarray*}
since the product is a law of composition.
\end{proof}

The next result concerns invertibility in $\overline{\Lambda}^{(1)}$. With a finite number of generators, a supernumber $z$ is invertible if and only if its body $z_B$ 
is different from $0$, as proved in Proposition \ref{invertibility-lambda}. Note that the proof follows easily from the fact that the soul $z_S$ is nilpotent. In $\overline{\Lambda}^{(1)}$, the latter does not hold in general. However, as the following theorem shows, $z_S$ is quasi-nilpotent and  invertibility is still characterized by $z_B\not=0$.

\begin{theorem}
\label{inv}
Let $z\in\overline{\Lambda}^{(1)}$. Then, $z$ is invertible in $\overline{\Lambda}^{(1)}$ if and only if $z_B\neq0$.
\end{theorem}

\begin{proof}
At first, assume $z\in\overline{\Lambda}^{(1)}_{even}$, which is a commutative Banach algebra. Then, let $\varphi$ be a homomorphism between $\overline{\Lambda}^{(1)}_{even}$ and $\mathbb{C}$, it is easy to prove that $\varphi(z)=z_B$. In fact, it follows from the fact that
\[
\varphi(i_\alpha)^2=\varphi(i_\alpha^2)=\varphi(0)=0
\]
and
\[
\varphi(1)=1.
\]
Thus by Gelfand's theorem  on invertibility in commutative Banach algebras 
(see e.g. \cite{dougtech}), 
$z$ is invertible if and only if $\varphi(z_B)\not=0$, proving our theorem in the case $z\in\overline{\Lambda}^{(1)}_{even}$. \smallskip

In the general case, note that if $z$ is invertible in $\overline{\Lambda}^{(1)}$, there exists $w\in\overline{\Lambda}^{(1)}$ such that $zw=wz=1$. In particular, $z_Bw_B=1$, which shows that $z_B\neq0$.

Conversely, without loss of generality let $z=1+u+v\in\overline{\Lambda}^{(1)}$, where $u+v$ is the soul of $z$ with $u$ being its even part and $v$ being the odd one. As already discussed, $1+u$ is invertible. Then,
\[
z=\left(1+u\right)\left[1+\left(1+u\right)^{-1}v\right],
\]
which is invertible if and only if $w=1+\left(1+u\right)^{-1}v$ is invertible. Note that $w_S\in\overline{\Lambda}^{(1)}_{odd}$. Thus to complete the proof of the theorem we only need to show that $z\in\overline{\Lambda}^{(1)}_{odd}$ is invertible if $z_B\neq0$. To do so, let $z=1+v$, where $v\in\overline{\Lambda}^{(1)}_{odd}$. Since by Proposition \ref{odd-square}, $v^2=0$, it follows trivially that $w=1-v$ is the inverse of $z$.
\end{proof}

\begin{theorem}
Let $\lambda\in\mathbb{C}$ be a complex variable and the power series
\begin{equation}
f(\lambda) = \sum_{n=0}^\infty c_n \lambda^n, \quad c_n\in\mathbb{C},
\label{power-series}
\end{equation}
be analytic in a neighborhood of the origin. If $z_S\in\overline{\Lambda}^{(1)}$ is a supernumber with zero body, then $f(z_S)$ converges in $\overline{\Lambda}^{(1)}$.
\label{ps-convergence}
\end{theorem}

\begin{proof}
Note that the power series referred in the statement of the theorem has a positive radius of convergence centered at the origin. Also, recall that the spectral radius formula asserts that
\[
\lim_{n\rightarrow\infty}\Vert z_S^n\Vert_1^{1/n} = \sup\left\{|x|\suchthat x\in\rho(z_S)\right\},
\]
where $\rho(z_S)$ is the spectral radius of $z_S$. Then, the power series we are considering converges if $\rho(z_S)=0$. But that is the case since, by Theorem \ref{inv}, if $\lambda\in\mathbb{C}$, $z_S-\lambda$ is not invertible if and only if $\lambda=0$.
\end{proof}

\begin{corollary}
Let $f(\lambda)$ be a complex analytic function for $\lambda\in\Omega$. Thus if $z\in\overline{\Lambda}^{(1)}$ is such that $z_B\in\Omega$, $f(z)$ converge in $\overline{\Lambda}^{(1)}$.
\label{f-extension}
\end{corollary}

\begin{proof}
Observe that we can (at least formally) write
\begin{equation}
f(z) = \sum_{n=0}^\infty \frac{f^{(n)}(z_B)}{n!} z_S^n.
\label{f-extension}
\end{equation}
In fact, by Theorem \ref{ps-convergence}, such a power series converges in $\overline{\Lambda}^{(1)}$.
\end{proof}

\begin{remark}
We note that formula \eqref{f-extension} usually appears in the literature as a formal power series -- see for instance \cite{MR1172996}. In Corollary \ref{f-extension}, 
it really has a meaning as an infinite sum, in the sense that it converges in $\overline{\Lambda}^{(1)}$.
\end{remark}

\begin{corollary}
In expression \eqref{power-series}, let $c_n\in\overline{\Lambda}^{(1)}$. If there exists a real number $\rho>0$ such that $f(\lambda)\in\overline{\Lambda}^{(1)}$ for every $0\le \lambda<\rho$, then
\[
f(z)=\sum_{n=0}^\infty c_n z^n \quad\text{and}\quad f(z)=\sum_{n=0}^\infty z^n c_n
\]
converge for every $z\in\overline{\Lambda}^{(1)}$ such that $|z_B|<\rho$.
\end{corollary}

\begin{proof}
Let $\lambda<\rho$ be a real number and $c_n = \sum_{\alpha\in\mathfrak{I}_0} {c_n}_\alpha i_\alpha$. We can write
\[
f(\lambda) = \sum_{n=0}^\infty c_n \lambda^n = \sum_{\alpha\in\mathfrak{I}_0}\left(\sum_{n=0}^\infty \lambda^n {c_n}_\alpha\right) i_\alpha.
\]
Moreover, note that there exist complex coefficients ${d_n}_\alpha(z_B)$ such that
\[
\sum_{n=0}^\infty z^n {c_n}_\alpha=\sum_{n=0}^\infty z_S^n {d_n}_\alpha(z_B) = \sum_{n=0}^\infty {d_n}_\alpha(z_B) z_S^n.
\]
Hence, by Theorem \ref{ps-convergence}, the power series
\[
f(z)=\sum_{n=0}^\infty c_n z^n = \sum_{n=0}^\infty d_n(z_B) z_S^n
\]
and
\[
f(z)=\sum_{n=0}^\infty z^n c_n = \sum_{n=0}^\infty z_S^n d_n(z_B)
\]
converge for every $z$ such that $z_B<\rho$.
\end{proof}

The next results concern the $k$-th root $z^{1/k}$ of a positive supernumber $z$. By $z^{1/k}$, we mean the analytic extension of $\lambda^{1/k}$ defined in $\mathbb{C}\setminus(-\infty,0]$.

\begin{corollary}
Let $z\in\overline{\Lambda}^{(1)}$ be an invertible supernumber. Then, there exists an element $w\in\overline{\Lambda}^{(1)}$ such that $z=w^k$, where $k\geq2$.
\label{k-root}
\end{corollary}

\begin{proof}
Recall that the complex function $f(\lambda)=\lambda^{1/k}$ is analytic for every complex $\lambda\neq0$. Then, by Corollary \ref{f-extension}, $f$ converges in $\overline{\Lambda}^{(1)}$ for every $z\in\overline{\Lambda}^{(1)}$ such that $z_B\neq0$.
\end{proof}

\begin{corollary}
Let $z\in\overline{\Lambda}^{(1)}$ be a positive supernumber. Then, $z$ has a superreal square root.
\label{positive-sqrt}
\end{corollary}

\begin{proof}
Let $z\in\overline{\Lambda}^{(1)}$ is a positive supernumber. By Corollary \ref{k-root}, $1+z_S/z_B$ has a square root since its body $z_B$ is different from zero.  Moreover, it is easy to check that $\sqrt{z}^\dag=\sqrt{z^\dag}$. Since $z^\dag=z$ in our case, we conclude that $\sqrt{z}^\dag=\sqrt{z}$, i.e., the square root of $z$ is superreal.
\end{proof}

\begin{remark}
{\rm Corollary \ref{positive-sqrt} implies that Proposition \ref{body-part} can be extended to supernumbers in $\overline{\Lambda}^{(1)}$.}
\end{remark}

One may also consider elements in the module $\left(\overline{\Lambda}^{(1)}\right)^{p\times q}$, i.e., $\overline{\Lambda}^{(1)}$-valued matrices, or supermatrices. This is what we explore in the next section.


\section{Algebra of $\overline{\Lambda}^{(1)}$-valued matrices}
\label{sec3}
The theory of matrices with entries in $\overline{\Lambda}^{(1)}$ present a number of difficulties. Since we are in a non-commutative setting we cannot resort to
results from e.g \cite{MR1409610}; on the other hand, and unlike the quaternions a spectral theorem for (appropriately defined) Hermitian matrices remains to be stated.
Still, the notion of positivity for Grassmann numbers introduced in Definition \ref{posi_num} allows us to develop the tools we need.\smallskip

We start this section extending the conjugation $\dag$ to matrices $M$ over the module $\left(\overline{\Lambda}^{(1)}\right)^{p\times q}$. This is done in the following way
\[
M^*=(m_{kj}^\dag).
\]
Note that
\[
(ML)^*=L^*M^*
\]
for matrices $M$ and $L$ of appropriate sizes.\smallskip

We note that the module $\left(\overline{\Lambda}^{(1)}\right)^{p\times q}$ is built by replacing the coefficients of a complex vector space of dimension $p\times q$ with supernumbers in $\overline{\Lambda}^{(1)}$. The usual way supermatrices are treated in the literature is from a construction of a module that uses a superspace as a ``basis''. In a superspace, one has $n_1$ even and $n_2$ odd basis vectors. Similarly to the case of supernumbers, the even supervectors behave like a regular vector. However, the odd ones require special attention. Moreover, it is standard to consider real even elements and imaginary odd ones. Such a choice leads to extra negative signs when computing the adjoint of a supervector. For more details, an introduction to supervector spaces is given in the first chapter of \cite{MR1172996}.\smallskip

As already mentioned, we do not consider supervector spaces in our approach. Equivalently, one could say we study the cases where $n_2=0$.

\begin{definition}
The norm of a matrix $M=(m_{jk})\in\left(\overline{\Lambda}^{(1)}\right)^{p\times q}$ is defined by
\begin{equation}
\Vert M\Vert_1 \equiv \sum_{j,k} \Vert m_{jk}\Vert_1.
\label{matrix-norm}
\end{equation}
\end{definition}

\begin{proposition}
Let $M\in\left(\overline{\Lambda}^{(1)}\right)^{p\times p}$, and $c,d\in\left(\overline{\Lambda}^{(1)}\right)^{p\times 1}$. Then, $d^*Mc=0$ for every $c$ and $d$ if and only if $c^*Mc=0$ for every $c$. Moreover, $M=0$ when either of these conditions hold.
\label{null-matrix}
\end{proposition}

\begin{proof}
Let $M=(m_{jk})\in\left(\overline{\Lambda}^{(1)}\right)^{p\times p}$. If $d^*Mc=0$ for every $c$ and $d$, then it is clear that $c^*Mc=0$ for every $c$ since this is just the particular case $d=c$. To prove the converse, observe that, by the polarization identity,
\begin{equation}
\label{polarization}
d^*Mc = \frac{1}{4}\sum_{k=0}^3 \left(c^*+(-i)^kd^*\right)M\left(c+i^kd\right).
\end{equation}
Because each term of the sum is of the type $a^*Ma$, where $a\in\left(\overline{\Lambda}^{(1)}\right)^{p\times 1}$, they are zero by assumption. Therefore, $d^*Mc=0$.

Finally, to see that such a $M$ is the null matrix we observe that if the $j$-th component of $d$ and $c$ are given by $d_j=\delta_{jr}$ and $c_j=\delta_{js}$, $r,s=1,\hdots,p$, then $d^*Mc=0$ becomes just $m_{rs}=0$.
\end{proof}

\begin{proposition}
Let $M\in\left(\overline{\Lambda}^{(1)}\right)^{p\times p}$ be a matrix such that $M_B$ is regular, meaning that all its main minor matrices are invertible. 
Then, $M$ can be factorized as $M=LDU$, where $D$ is a diagonal matrix composed by invertible supernumbers, and $L$ and $U$ are, respectively, lower and upper triangular matrices with main diagonals composed by ones.
\label{decomposition}
\end{proposition}

\begin{proof}
Denote by $m_{jk}$ the elements of $M$ and assume $M_B>0$. Then, $\left(m_{kk}\right)_B>0$ for every $k\in\left\{1,2,\hdots,p\right\}$, which implies that $m_{kk}$ is invertible. We can, then, write
\[
M=\left(
\begin{array}{c c}
m_{11} & B \\
C      & E
\end{array}
\right) = \left(
\begin{array}{c c}
1              & 0 \\
Cm_{11}^{-1} & I_{p-1}
\end{array}
\right)
\left(
\begin{array}{c c}
m_{11} & 0 \\
0      & E-Cm_{11}^{-1}B
\end{array}
\right)
\left(
\begin{array}{c c}
1 & m_{11}^{-1}B \\
0 & I_{p-1}
\end{array}
\right),
\]
where $B$, $C$, and $E$ are block matrices. By successively repeating this process, one concludes that $M=LDU$, where $L$, $D$, and $U$ are 
described in the statement of the proposition. The iterations cab be done because at each step the $(1,1)$ entry of the Schur complment is invertible since the
corresponding main minor is invertible.
\end{proof}

\begin{corollary}
Any matrix $M\in\left(\overline{\Lambda}^{(1)}\right)^{p\times p}$ is invertible if and only if its body is invertible.
\end{corollary}

\begin{proof}
It is clear that an invertible matrix $M\in\left(\overline{\Lambda}^{(1)}\right)^{p\times p}$ has an invertible body. We shall, then, only prove the converse. To do so, assume $M_B$ is invertible and consider $A=MM_B^*$. Note that $A_B>0$ and is, in particular, regular. Then, by Proposition \ref{decomposition}, $A$ can be decomposed as $A=LDU$, where $D$ is a diagonal matrix with body composed by positive real numbers, and $L$ and $U$ are, respectively, lower and upper triangular matrices with main diagonals composed by ones. Observe that $D$ is invertible because positive supernumbers are invertible. Moreover, $L$ and $U$ are also invertible -- as is any lower or upper triangular matrix with main diagonal composed by ones. Therefore, $A$ is invertible and $M^{-1}=\left(M_B\right)^{-*}A^{-1}$.
\end{proof}

\begin{corollary}
Let $\lambda\in\mathbb{C}$ be a complex variable and the power series
\[
f(\lambda) = \sum_{n=0}^\infty c_n \lambda^n, \quad c_n\in\mathbb{C},
\]
be analytic in a neighborhood of the origin. If $M_S\in\left(\overline{\Lambda}^{(1)}\right)^{p\times p}$ is a supermatrix with zero body, then $f(M_S)$ converges in $\left(\overline{\Lambda}^{(1)}\right)^{p\times p}$.
\label{ps-convergence-matrix}
\end{corollary}

The proof of the above corollary follows in a similar way as the proof of Corollary \ref{ps-convergence} and, because of it, is omitted.

\begin{definition}
Let $M\in\left(\overline{\Lambda}^{(1)}\right)^{p\times p}$. Then, $M$ is said to be a non-negative supermatrix -- or simply super non-negative  -- if
\begin{equation}
c^*Mc\succeq0 \label{superpositivity}
\end{equation}
for every $c\in\left(\overline{\Lambda}^{(1)}\right)^{p\times 1}$. By the above inequality, we mean that the product $c^*Mc$ is a non-negative supernumber. 
Moreover, we say that $M$ is superpositive if
\begin{equation}
c^*Mc\succ0 \label{stric-superpositivity}
\end{equation}
for every $c\in\left(\overline{\Lambda}^{(1)}\right)^{p\times 1}$ such that $c_B$ is not the null element.
\label{matrix-positivity}
\end{definition}

\begin{proposition}
A supermatrix $M\in\left(\overline{\Lambda}^{(1)}\right)^{p\times p}$ is non-negative if and only if it is self-adjoint and its body is non-negative.
\label{body-reduction-positive}
\end{proposition}

\begin{proof}
First, let $M\in\left(\overline{\Lambda}^{(1)}\right)^{p\times p}$ be non-negative. Then by definition, \eqref{superpositivity} holds. Because of Proposition \ref{positive-real}, it implies that $c^*Mc$ is superreal for every $c\in\left(\overline{\Lambda}^{(1)}\right)^{p\times 1}$, i.e.,
\[
c^*Mc = c^*M^*c \Rightarrow c^*(M-M^*)c=0
\]
for every $\left(\overline{\Lambda}^{(1)}\right)^{p\times 1}$. Then, by Proposition \ref{null-matrix}, $M$ is self-adjoint, i.e., $M^*=M$. Moreover, by Proposition \ref{body-part}, equation \eqref{superpositivity} implies that
\begin{equation}
c^*_BM_Bc_B \geq 0,
\label{pd-body}
\end{equation}
for every $c\in\left(\overline{\Lambda}^{(1)}\right)^{p\times 1}$. It means that the body of $M$ is non-negative.

Conversely, let $M$ be self-adjoint and such that $M_B$ is non-negative, i.e., assume \eqref{pd-body} holds. Those facts combined with the extension of Proposition \ref{body-part} to supernumbers in $\overline{\Lambda}^{(1)}$ imply that \eqref{superpositivity} holds.
\end{proof}

The next two results follow easily and their proof are omitted.

\begin{proposition}
The sum of two non-negative supermatrices is non-gegative. Similarly, the sum of two positive supermatrices is superpositive.
\end{proposition}

\begin{corollary}
Every superpositive matrix is invertible.
\end{corollary}

\begin{theorem}
Let $M\in\left(\overline{\Lambda}^{(1)}\right)^{p\times p}$. Then, the following statements are equivalent:
\begin{itemize}
\item[(i)]   $M$ is superpositive;
\item[(ii)]  $M$ is self-adjoint and $M_B$ is positive;
\item[(iii)] $M=LDU$, where $D$ is a superpositive diagonal matrix, $L$ is a lower triangular matrix with main diagonal composed by ones, and $U=L^*$;
\item[(iv)]  $M=LL^*=UU^*$, where $L,U\in\left(\overline{\Lambda}^{(1)}\right)^{p\times p}$, and $L$ and $U$ are, respectively, lower and upper triangular matrices and their main diagonals are composed by ones.
\end{itemize}
\label{equiv-def-stric-positive}
\end{theorem}

\begin{proof}
We divide the proof into a number of steps.\\

STEP 1: {\sl (i) implies (ii).}\smallskip

This follows as a corollary from Proposition \ref{body-reduction-positive}. \\

STEP 2: {\sl (ii) implies (iii).}\smallskip

It follows easily from Proposition \ref{decomposition}. We just note that in a similar way one can also show that $M=U'D'L'$, where $L'=U'^*$.\\

STEP 3: {\sl (iii) implies (iv).}\smallskip

Observe that the matrix $D$ in $M=LDU$ is composed by positive supernumbers. By Corollary \ref{positive-sqrt}, those supernumbers have positive square root. 
Then, $D$ also has a superpositive diagonal square root. Denote it by $S$, so $D=S^2$ and
\[
M=LDU=LS^2D=\left(LS\right)\left(LS\right)^* = L'L'^*.
\]
In a similar way, one could start from a decomposition $M=UDL$ and conclude that $M=U'U'^*$.\\

STEP 4: {\sl (iv) implies (i).}\smallskip

The proof follows from the fact that for every $c\in\left(\overline{\Lambda}^{(1)}\right)^{p\times 1}$ such that $c_B$ is not the null element
\[
c^*Mc=c^*LL^*c=\left(L^*c\right)^*\left(L^*c\right)\succ0
\]
or, similarly, $c^*UU^*c\succ0$.
\end{proof}


\section{Extension of Toeplitz matrices}
\setcounter{equation}{0}
\label{sec-matrix-ext}
Recall that a self-adjoint matrix with a constant main diagonal is called a Toeplitz matrix. As an application of the results and definitions of the previous section, 
the aim of this section is to solve an extension problem for such matrices.

\begin{problem}
Given a superpositive Toeplitz matrix
\[
T_N = \left(
\begin{array}{c c c c}
r_0      & r_1          & \hdots & r_N     \\
r_1^\dag & r_0          & \hdots & r_{N-1} \\
\vdots   & \vdots       & \ddots & \vdots  \\
r_N^\dag & r_{N-1}^\dag & \hdots & r_0
\end{array}
\right),
\]
where $r_0,\cdots,r_N\in\overline{\Lambda}^{(1)}$, $r_0^\dag=r_0$, how can one create a superpositive Toeplitz 
extension $T_{N+1}$ of it? In other words, what is the condition that must be satisfied by a supernumber $r_{N+1}\in\overline{\Lambda}^{(1)}$ such that
\[
T_{N+1} = \left(
\begin{array}{c c}
T_N       & b_{N+1} \\
b_{N+1}^* & r_0
\end{array}
\right),
\]
where $b_{N+1}$ is a column element with coordinates $b_{N+1}=\left(r_{N+1},r_{N},\cdots,r_1\right)\equiv\left(r_{N+1},b_{N}\right)$, is superpositive?
\label{problem-sec2}
\end{problem}

The counterpart of this problem in the complex domain appears as an important player in many areas, such as signal and image processing, system control, and in prediction of stationary processes of second order. Moreover, the center of a one step extension in the classical theory is related to the concept of maximum entropy. The center is also directly associated to the best estimation of parameters in stochastic processes, making it -- and the inversion of Toeplitz matrices -- an important question when solving the Yule-Walker equations, see e.g. \cite{van-den-bos}.\smallskip

Going back to the case where complex coefficients are replaced by $\overline{\Lambda}^{(1)}$-valued ones, we do not present here an analogous of such applications. Instead, our focus in this section is exclusively to answer Problem \ref{problem-sec2}. We first write
\[
T_{N+1} = \left(
\begin{array}{c c}
1 & b_{N+1}^*T_N^{-1} \\
0 & I
\end{array}
\right)
\left(
\begin{array}{c c}
r_0-b_{N+1}^*T_N^{-1}b_{N+1} & 0 \\
0                            & T_N
\end{array}
\right)
\left(
\begin{array}{c c}
1               & 0 \\
T_N^{-1}b_{N+1} & I
\end{array}
\right).
\]
Then, by Theorem \ref{equiv-def-stric-positive}, $T_{N+1}$ is a superpositive matrix if and only if $r_0-b_{N+1}^*T_N^{-1}b_{N+1}$ is a positive supernumber. Writing
\[
T_N = \left(
\begin{array}{c c}
r_0   & a_N \\
a_N^* & T_{N-1}
\end{array}
\right) = \left(
\begin{array}{c c}
1 & a_NT_{N-1}^{-1} \\
0 & I
\end{array}
\right)
\left(
\begin{array}{c c}
r_0-a_NT_{N-1}^{-1}a_N^* & 0 \\
0                        & T_{N-1}
\end{array}
\right)
\left(
\begin{array}{c c}
1                 & 0 \\
T_{N-1}^{-1}a_N^* & I
\end{array}
\right),
\]
we conclude that
\[
r_0-b_{N+1}^*T_N^{-1}b_{N+1}\succ0\Leftrightarrow \left(r_{N+1}-c_N\right)^\dag\alpha\left(r_{N+1}-c_N\right)\prec r_0-b_N^*T_N^{-1}b_N,
\]
where $\alpha=\left(r_0-a_NT_{N-1}^{-1}a_N^*\right)^{-1}$ and $c_N=a_NT_{N-1}^{-1}b_N$. Furthermore, defining $\xi^\dag\xi\equiv r_0-b_N^*T_N^{-1}b_N$, there exists $\eta\prec1$ such that we can rewrite the above expression on the right-hand side as
\[
\left(r_{N+1}-c_N\right)^\dag\alpha\left(r_{N+1}-c_N\right)\prec\xi^\dag\xi
\]

In case $\xi$ is a real supernumber, we have
\[
\alpha^{1/2}\left(r_{N+1}-c_N\right)\prec\xi \Rightarrow r_{N+1} = c_N +\alpha^{-1/2}\eta\xi.
\]
The last part of the previous expression can be seen as the natural definition of a disk with center in $c_N$, left radius $\alpha^{-1/2}$, and right radius $\xi$. We call it a superdisk. \smallskip

Note that the ``geometry'' induced by the definition of positivity in the Grassmann algebra is different from the one induced by the $1$-norm. In fact, elements in the superdisk in general do not have a bounded norm. For instance, consider
\[
a=\frac{1}{2}\left(1+\lambda i_1\right) \quad \lambda\in\mathbb{C}.
\]
Even though it is clearly inside the superdisk $zz^\dag\prec 1$, its norm is $\Vert a\Vert_1=(1+|\lambda|)/2$, which can be arbitrarily large. Such a problem has analogous in the complex setting, where it is an instance of a large family of extension problems and can be solved using the band method -- see \cite{gkw-89,MR90c:47022}.


\section{Realization theory and rational functions}
\setcounter{equation}{0}
\label{ratio}
The main goal of this section is to define and study realization theory and rational functions for {\sl a priori} formal power series $F$ analytic at the origin with coefficients in $\left(\overline{\Lambda}^{(1)}\right)^{p\times p}$
\begin{equation}
F(z)=\sum_{n=0}^\infty z^n f_n,
\label{function-in-gamma}
\end{equation}
where the variable $z$ varies in a neighborhood of the origin. Before doing so, we define a product on such a set.

\begin{definition}
Let $F$ and $G$ be two power series of the type \eqref{function-in-gamma}. The Cauchy (or star) product is defined as
\begin{equation}
F\star G(z) \equiv \sum_{n\in\mathbb{Z}} z^n \left(\sum_{u\in\mathbb{Z}} f_u g_{n-u}\right).
\label{star-product}
\end{equation}
\label{def-star-product}
\end{definition}

Note that such a product reduces to the regular product in the case of $z=\lambda\in\mathbb{C}$. Moreover, for $F(z)$ invertible, the star and the regular products can be related as follows
\begin{eqnarray*}
F\star G(z) & = & F(z)\star \sum_{n\in\mathbb{Z}} z^n g_n \\
            & = & \sum_{n\in\mathbb{Z}} z^n F(g)g_n \\
            & = & F(z)\sum_{n\in\mathbb{Z}} F(z)^{-1}z^n F(z)g_n \\
            & = & F\left(z\right)G\left(F(z)^{-1}zF(z)\right).
\end{eqnarray*}
For this formula in the setting of slice hyperholomorphic functions see 
\cite{MR2752913}.

\begin{definition}
A function $F\in\Gamma^{p\times q}(\Omega)$ defined from a neighborhood of the origin in $\overline{\Lambda}^{(1)}$ is said to admit a realization if it can be represented as
\begin{equation}
\label{real123}
F(z)=D+zC\star (I_N-zA)^{-\star}B,
\end{equation}
where $D=F(0)$ and $A,B,C$ are super matrices with entries in $\Lambda$ and of appropriate sizes.
\end{definition}

Expression \eqref{real123} is called a realization of $F$. The notion of a realization originated with linear system theory -- see e.g. \cite{MR0255260}. A realization is many times represented by the block matrix
\begin{equation}
\label{real456}
\begin{pmatrix}
A&B\\ C&D
\end{pmatrix}.
\end{equation}

Rational functions are functions defined at the origin by an expression of the form \eqref{real123}. They play an important role in analysis and related topics. For instance, a finite Blaschke product, which we study in Section \ref{sec-rkpm}, is a rational function of a special type. Before focusing our attention on rational functions, we study properties of realizations. To start, we present well known formulas on the realization for products, sums, concatenations and inverses. To prove these results, first one takes $z=\lambda\in\mathbb C$. The arguments are then the same as in the case of complex coefficients. After that, one just replaces the pointwise product by the star product and $\lambda$ by a $\overline{\Lambda}^{(1)}$-valued variable $z$. Below, we give the proof of equation \eqref{leminv1} and leave the proofs of the formulas in Lemma \ref{lemsumprod} to the reader.

\begin{lemma}
\label{leminv}
Assume that $D$ in \eqref{real123} is invertible. Then,
\begin{equation}
F(z)^{-\star}=D^{-1}-zD^{-1}C\star (I-zA^{\times})^{-\star}BD^{-1}
\label{leminv1}
\end{equation}
is a realization of $F^{-\star}$, with
\begin{equation}
A^\times=A-BD^{-1}C.
\end{equation}
\end{lemma}

\begin{proof}
As already mentioned, we first consider the case $z=\lambda\in\mathbb{C}$. Then, equation \eqref{leminv1} becomes
\[
F(z)^{-1}=D^{-1}-zD^{-1}C(I_N-zA^{\times})^{-1}BD^{-1}.
\]

Recall the formula
\begin{equation}
\label{invinv}
(I_n+ab)^{-1}=I-a(I_m+ba)^{-1}b,
\end{equation}
where $a$ and $b$ are matrices of appropriate sizes in an algebra which contains the complex numbers. In the above formula, we set
\[
a=\lambda D^{-1}C\quad {\rm and}\quad b=(I_N-\lambda A)^{-1}B
\]
for values of $\lambda$ such that the inverse $(I_N-\lambda A)^{-1}$ exists. Therefore,
\[
\begin{split}
F(\lambda)^{-1}&=\left[D\left(I+\lambda D^{-1}C(I_N-\lambda A)^{-1}B\right)\right]^{-1}\\
      &=\left[I-\lambda D^{-1}C\left(I_N-(I_N-\lambda A)^{-1}B\lambda D^{-1}C\right)^{-1}(I_N-\lambda A)^{-1}B\right]D^{-1}\\
      &=D^{-1}- \lambda D^{-1}C\left(I_N+(I_N-\lambda A)^{-1}B\lambda D^{-1}C\right)^{-1}(I_N-\lambda A)^{-1}BD^{-1}\\
      &=D^{-1}- \lambda D^{-1}C\left(I_N-\lambda A+\lambda B D^{-1}C\right)^{-1}BD^{-1}\\
      &=D^{-1}-D^{-1}C(I_N-\lambda A^\times)^{-1}BD^{-1},
\end{split}
\]
which was to be shown. As already mentioned, the proof of \eqref{leminv1} follows similarly by replacing $\lambda$ by a generic $\overline{\Lambda}^{(1)}$-valued variable $z$ and the pointwise product by the star product.
\end{proof}

\begin{lemma}
\label{lemsumprod}
Let \begin{equation}
F_j(z)=D_j+zC_j\star (I_{N_j}-zA_j)^{-\star}B_j,\quad j=1,2
\end{equation}
be two realizations of compatible sizes. Then:\\
$(1)$ A realization of $F_1(z)F_2(z)$ is given by
\begin{equation}
A=\begin{pmatrix}A_1&B_1C_2\\ 0&A_2\end{pmatrix},\quad B=\begin{pmatrix} B_1D_2\\ B_2\end{pmatrix},\quad C\begin{pmatrix} C_1& D_1C_2\end{pmatrix},\quad D=D_1D_2.
\end{equation}
$(2)$ A realization of $F_1+F_2$ is given by
\begin{equation}
A=\begin{pmatrix}A_1&0\\ 0&A_2\end{pmatrix},\quad B=\begin{pmatrix} B_1\\ B_2\end{pmatrix},\quad C=\begin{pmatrix} C_1& C_2\end{pmatrix},\quad D=D_1+D_2.
\end{equation}
$(3)$ A realization of $\begin{pmatrix}F_1&F_2\end{pmatrix}$ is given by
\begin{equation}
A=\begin{pmatrix}A_1&0\\ 0&A_2\end{pmatrix},\quad B=\begin{pmatrix} B_1&0\\0& B_2\end{pmatrix},\quad C=\begin{pmatrix} C_1& C_2\end{pmatrix},
\quad D=\begin{pmatrix}D_1& D_2\end{pmatrix}
\end{equation}
$(4)$ A realization of $\begin{pmatrix}F_1\\F_2\end{pmatrix}$ is given by
\begin{equation}
A=\begin{pmatrix}A_1&0\\ 0&A_2\end{pmatrix},\quad B=\begin{pmatrix} B_1\\B_2\end{pmatrix},\quad C=\begin{pmatrix} C_1&0\\ 0& C_2\end{pmatrix},
\quad D=\begin{pmatrix}D_1\\ D_2\end{pmatrix}.
\end{equation}
\end{lemma}

A $k$-th order polynomial in $\Gamma^{p\times q}$ is a finite sum of the form
\begin{equation}
\label{poly}
M(z)=M_0+zM_1+\cdots +z^kM_k
\end{equation}
where $M_0,\ldots, M_k\in\left(\overline{\Lambda}\right)^{p\times q}$.

\begin{lemma}
\label{polreal}
Any polynomial in $z$ admits a realization.
\end{lemma}

\begin{proof}
In view of Lemma \ref{lemsumprod}, it suffices to prove that constant terms and terms of the form $zM$ admit realizations. But this is clear. Indeed,
a constant supermatrix $M\in\left(\overline{\Lambda}\right)^{p\times q}$ corresponds to the realization $A=B=C=0$ and $D=M$, while the function $zM$ correponds to $C=M$, $A=D=0$, $B=I_q$.
\end{proof}

We say $(C,A)$ is an observable pair of matrices if
\begin{equation}
\label{obs}
\cap_{u=0}^\infty \ker CA^u=\left\{0\right\}.
\end{equation}
Moreover, the pair $(A,B)$ is said to be controllable if
\begin{equation}
\xi(I-\lambda A)^{-1}B\equiv 0\,\,\Longrightarrow\,\, \xi=0,
\label{cont}
\end{equation}
where $\lambda\in\mathbb{R}$.

\begin{definition}
The realization \eqref{real123} is called minimal if the pair $(C,A)$ is observable and the pair $(A,B)$ is controllable.
\end{definition}

\begin{proposition}
Two minimal realizations are similar.
\end{proposition}

\begin{proof}
Assume
\[
F(z)=D_j+zC_j\star (I_{N_j}-zA_j)^{-\star}B_j, \quad j=1,2,
\]
are two minimal realizations of $F$. Then $D_1=D_2$. Note that it is possible, then, to write
\[
\frac{F(x)-F(y)}{x-y}=C_1(I_{N_1}-xA_1)^{-1}(I_{N_1}-yA_1)^{-1}B_1=C_2(I_{N_2}-xA_2)^{-1}(I_{N_2}-yA_2)^{-1}B_2,
\]
and $x,y\in\mathbb{R}$, with the understanding that the left-side is equal to
$F^\prime(x)$ if $x=y$.\\

We also define the following operators $U$ and $V$:
\[
U\left((I_{N_1}-yA_1)^{-1}B_1\xi\right)=(I_{N_2}-yA_2)^{-1}B_2\xi,
\]
\[
V\left((I_{N_2}-yA_2)^{-1}B_2\xi\right)=(I_{N_1}-yA_1)^{-1}B_1\xi,
\]
where $\xi\in\left(\overline{\Lambda}^{(1)}\right)^q$. The fact that the pairs $(C_j,A_j)$ are observable assures that such operators are well defined. From the above definitions and because the pairs $(A_j,B_j)$ are controllable, we obtain
\[
UV=\mathbb{I}_{N_2},
\]
\[
VU=\mathbb{I}_{N_1},
\]
where $\mathbb{I}_N$ denotes the identity operator.

Now, we show that the operators $U$ and $V$ admit a matrix representation. To do so, we define elements $e_j$, $j=1,\hdots,N$, of modules $\left(\overline{\Lambda}^{(1)}\right)^N$ with components characterized by
\[
\left(e_j\right)_k = \delta_{jk}.
\]
Writing
\[
(I_{N_1}-yA_1)^{-1}B_1\xi = \sum_{j=1}^{N_1} e_j \alpha_j,
\]
where $\alpha_j\in\overline{\Lambda}^{(1)}$,
\[
U\left(\sum_{j=1}^{N_1} e_j \alpha_j\right) = \sum_{j=1}^{N_1} U\left(e_j\right) \alpha_j = \sum_{j=1}^{N_1} \left(\sum_{k=1}^{N_2}u_{kj}e_j\right) \alpha_j.
\]
Hence, $U$ can be represented by a $N_2\times N_1$ matrix $\tilde{U}=\left(u_{kj}\right)$. Similarly, $V$ can be also represented by $N_1\times N_2$ matrix $\tilde{V}$. Moreover, because $(A_1,B_1)$ is controllable, $N_2\leq N_1$. On the other hand, the fact $(A_2,B_2)$ is also controllable implies that $N_1\leq N_2$.
\end{proof}

\begin{proposition}
If a realization \eqref{real123} is minimal, then the size of $A$ is minimal.
\end{proposition}

Having presented some definitions and results on realizations, we are now ready to discuss rational functions. First, we give for a function in $\Gamma^{p\times q}$ a number of equivalent definitions of a rational function, namely:
\begin{enumerate}
\item In terms of multiplication and inversion of polynomials with respect to the $star$-product;
\item In terms of $\star$ multiplication and $\star$ inversion of matrix polynomials;
\item In terms of a realization, as defined in the theory of linear systems;
\item In terms of the Taylor coefficients of the function;
\item In terms of the backward-shift operator;
\end{enumerate}

The equivalence between the various definitions is presented at the end of the section.

\begin{definition}
\label{def1}
A function in $\Gamma$ is rational if it can be obtained by a finite number of $\star$-inversions and $\star$-multiplications starting with constants and $z$.
In case the functions live in $\Gamma^{p\times q}$, it is called rational if all its entries have the above property.
\end{definition}

\begin{definition} (definition in terms of product of polynomials)
A rational function $F\in\Gamma^{p\times r}$ invertible at $0$ is a finite product of the type
\begin{equation}
\label{M1M2M3}
F(z) = M_1(z) \star M_2(z)^{-\star} \star M_3(z),
\end{equation}
where $M_1\in\Gamma^{p\times q}$, $M_2\in\Gamma^{q \times q}$, and $M_3\in\Gamma^{q \times r}$ are polynomials.
\label{def2}
\end{definition}

\begin{definition} (realization definition)
A function $F\in\Gamma^{p\times q}$ is called rational if it can be represented as
\eqref{real123}.
\label{def3}
\end{definition}

\begin{definition} (Taylor coefficients)
Let $F\in\Gamma^{p\times q}$. It is rational if
\begin{equation}
f_n=\begin{cases}\, D\,\,\quad\hspace{7.5mm} if\,\,n=0\\
                    CA^{n-1}B\,\, if\,\, n\ge 1
\end{cases}.
\end{equation}
\label{def4}
\end{definition}

If $F\in\Gamma^{p\times q}$, define backward-shift operator $R_0$ by
\begin{equation}
R_0F(z)=f_1+zf_2+\hdots,
\end{equation}
where $F$ is of the form \eqref{function-in-gamma}. Note that for $z=\lambda\in\mathbb C$ we can write
\begin{equation}
R_0F(\lambda)=\left\{
\begin{array}{l l}
\frac{F(\lambda)-F(0)}{\lambda} & \text{if } \lambda \neq 0 \\
f_1                             & \text{if } \lambda=0
\end{array}
\right..
\label{rorororo}
\end{equation}

\begin{definition}
Let $F\in\Gamma^{p\times q}$. It is rational if the module spanned by $\{R_0^nFc\}$ is finitely generated, when $n$ runs through $\mathbb N$ and $c$ runs through $\left(\overline{\Lambda}^{(1)}\right)^q$.
\label{def6}
\end{definition}

\begin{theorem}
The five definitions given above of a rational function -- i.e., definitions \ref{def1}-\ref{def6} -- are equivalent.
\end{theorem}

\begin{proof} We divide the proof into a number of steps.\\

STEP 1: {\sl Definition \ref{def1} is equivalent to Definition \ref{def3}.}\smallskip

By Lemma \ref{polreal}, polynomials admit realizations. Moreover, each entry of a matrix-valued function obtained as in Definition \ref{def1} admits a realization
-- see Lemma \ref{lemsumprod} and the first item of Lemma \ref{leminv}. The fact that a matrix-valued function itself admits a realization follows from Lemma \ref{lemsumprod}.\\

STEP 2: {\sl Definition \ref{def2} is equivalent to Definition \ref{def3}.}\smallskip

The direct implication follows from Lemma \ref{leminv} and Lemma \ref{lemsumprod} and from the fact that polynomials admit realizations (Lemma \ref{polreal}).
Conversely, note that \eqref{real123} is a special case of \eqref{M1M2M3}. In fact, one can see it with the choices
\[
M_1(z)=\begin{pmatrix} D&zC\end{pmatrix},\quad M_2(z)=\begin{pmatrix}I&0\\0&I-zA\end{pmatrix},\quad \text{and} \quad M_3=\begin{pmatrix}I\\ B\end{pmatrix}.
\]

STEP 3: {\sl Definition \ref{def3} is equivalent to Definition \ref{def4}.}\smallskip

By the definition of the star product, we have
\[
F(z)=D+\sum_{n=1}^\infty z^nCA^{n-1}B 
\]
and hence Definition \ref{def4} holds. To prove the converse, it suffices to compute the converging series
\[
D+\sum_{n=1}^\infty \lambda^nCA^{n-1}B=D+\lambda C(I-\lambda A)^{-1}B,
\]
where $\lambda\in\mathbb{C}$, and consider its extension to a $\overline{\Lambda}^{(1)}$-valued variable $z$.\\

STEP 4: {\sl Definition \ref{def3} and Definition \ref{def6}
are equivalent.}\smallskip

Again, we restrict $z$ to be a complex variable $z=\lambda\in\mathbb C$. An induction gives
\[
(R_0^nF)(\lambda)=C(I-\lambda A)^{-1}A^{n-1}B,\quad n=1,2,\ldots
\]
and so the module generated by the $R_0^{n}F$ is included in the span of the columns of the matrix-function $C(I-\lambda A)^{-1}$. This span is a finite dimensional vector space, and 
when we extend $C(I-\lambda A)^{-1}$ to the Grassmann variable, we get a finitely generated module.\smallskip

Conversely, let $F$ be a supermatrix whose columns are composed by a generating set of functions. Then, there exists a matrix $A$ such that
\[
R_0F(z)=F(z)A
\]
In view of \eqref{rorororo} we have
\[
F(\lambda)-F(0)=\lambda F(\lambda)A,
\]
from which we get $F(\lambda)$ and its extension $F(z)$.\\

STEP 5: {\sl Definition \ref{def3} implies Definition \ref{def1}.}\smallskip

This follows directly from the definition of a realization.
\end{proof}

In the setting of complex numbers and Schur analysis, one usually adds metric conditions to a rational function, for instance, being contractive in the open unit disk, or taking unitary values (with respect to a possibly indefinite metric) on the unit circle. It is of interest to translate these metric conditions into conditions on a given realization of the function. In this work, such questions are addressed in Sections \ref{sec-schur} and \ref{sec-rkpm}.\smallskip

In preparation to Section \ref{sec-rkpm}, we introduce the following definition.

\begin{definition}
\label{rationalpotapov}
Let $J\in\Lambda^{n\times n}$ be a signature matrix, meaning that it is both self-adjoint and unitary. A rational function $U$ in $\Gamma$ is said to be $J$-unitary -- or symplectic -- if
\[
U(z)JU(z^{-\dag})^*=J,
\]
where the expression makes sense.
\end{definition}

\section{Rational Schur-Grassmann functions}
\setcounter{equation}{0}
\label{sec-schur}

In the classical setting, Schur functions can be charaterized in a number of equivalent ways (for instance as contractive multipliers of the Hardy space). In the rational setting, a Schur function $S$ can be defined as a matrix-valued rational function which is analytic at infinity with minimal realization given by
\begin{equation}
S(\lambda) = D+C(\lambda I- A)^{-1}B
\label{rational-analytic-infinity}
\end{equation}
such that
\begin{equation}
\left(
\begin{array}{c c}
A & B \\
C & D
\end{array}
\right)^*
\left(
\begin{array}{c c}
H & 0 \\
0 & I
\end{array}
\right)
\left(
\begin{array}{c c}
A & B \\
C & D
\end{array}
\right)\le
\left(
\begin{array}{c c}
H & 0 \\
0 & I
\end{array}
\right)
\label{schur-condition}
\end{equation}
for some (not uniquely defined) $H< 0.$ This is the positive real lemma, also known as Kalman-Yakubovich-Popov theorem; see \cite{Anderson_Moore,faurre,MR525380} and its generalization \cite{DDGK}. For the description of all $H$ satisfying \eqref{schur-condition} see e.g. \cite{faurre,MR525380}.\smallskip

The reader should note that \eqref{rational-analytic-infinity} has a different expression for its realization than the rational functions studied in the previous section, which can be written as
\begin{equation}
S(\lambda) = H + \lambda G(I-\lambda T)^{-1} F
\label{realization-auxiliar}
\end{equation}
and, in general, is not analytic at infinity. However, we just remind that both expressions are in fact equivalent if the matrix $A$ is invertible. In fact, let $H=D-CA^{-1}B$, $G=-CA^{-1}$, $T=A$, and $F=AB$. This allows to rewrite \eqref{realization-auxiliar} as
\begin{eqnarray*}
S(\lambda) & = & D-CA^{-1}B - \lambda CA^{-1} (I-\lambda A)^{-1} AB \\
           & = & D-CA^{-1}B + CA^{-1} (\lambda A-I)^{-1}(I+\lambda A-I) B \\
           & = & D + CA^{-1} (\lambda A-I)^{-1} B \\
           & = & D + C(\lambda I-A)^{-1} B.
\end{eqnarray*}

With the clarification for the use of a realization with a different expression, we take \eqref{rational-analytic-infinity} conditioned to \eqref{schur-condition} as the basis for our definition of a rational Schur-Grassmann function.

\begin{definition}
The $(\overline{\Lambda}^{(1)})^{p\times q}$-valued rational function  with realization
\[
S(z) = D+C\star(zI_N-A)^{-\star}B
\]
will be called a Schur-Grassmann function if there exists an Hermitian strictly negative matrix $H
\in (\overline{\Lambda}^{(1)})^{N\times N}$ (i.e., $H\prec 0$) such that
\begin{equation}
\left(
\begin{array}{c c}
A & B \\
C & D
\end{array}
\right)^*
\left(
\begin{array}{c c}
H & 0 \\
0 & I
\end{array}
\right)
\left(
\begin{array}{c c}
A & B \\
C & D
\end{array}
\right)\preceq
\left(
\begin{array}{c c}
H & 0 \\
0 & I
\end{array}
\right).
\label{expression-def-schur}
\end{equation}
\end{definition}

\begin{proposition}
\label{propo_equiv}
Let $S$ be a $(\overline{\Lambda}^{(1)})^{p\times q}$-valued rational function. Then, $S$ is a Schur-Grassmann function if and only if its body part $S_B$ is a Schur function.
\end{proposition}

\begin{proof}
This follows from Proposition \ref{body-reduction-positive} applied to \eqref{expression-def-schur}.
\end{proof}

\begin{remark}
{\rm Proposition \ref{propo_equiv} allows us to translate directly to the Schur-Grassmann functions a number of properties of Schur functions -- as we show next. However, not every result follows from a simple reduction to the body of the function. Interpolation problems are an example of such results, as can be seen in the solution of the Nevanlinna-Pick interpolation in Section \ref{sec-interpolation}.}
\end{remark}

The corollaries presented in the rest of this section follow from Proposition \ref{propo_equiv}.

\begin{corollary}
Let $S$ be a $(\overline{\Lambda}^{(1)})^{p\times q}$-valued rational function. Then, $S$ is a Schur-Grassmann function if and only $S(z)S(z)^*\preceq I$ for $zz^\dag\preceq 1$.
\label{schur-ineq-mtx}
\end{corollary}



\begin{corollary}
Let $S$ be a $(\overline{\Lambda}^{(1)})^{p\times q}$-valued  rational function. Then $S$ is a Schur-Grassmann function if and only if the kernel
\begin{equation}
\sum_{n=0}^\infty z^n\left(I-S(z)S(w)^*\right)(w^\dag)^n
\label{spd-expression}
\end{equation}
is superpositive.
\label{schur-spd}
\end{corollary}


We now introduce an Hermitian form, needed for our next corollary. For a power series given by \eqref{function-in-gamma} with coefficients in $\left(\overline{\Lambda}^{(1)}\right)^{p\times q}$, we say that $F\in\Gamma^{p\times q}$ if
\[
\left[F,F\right]\in\left(\overline{\Lambda}^{(1)}\right)^{q\times q},
\]
where the form $\left[\cdot,\cdot\right]$ is defined similarly to the one introduced in \eqref{form-wiener}:
\begin{equation}
\left[F,G\right]\equiv\sum_{n=0}^\infty g_n^* f_n.
\label{def-form}
\end{equation}

We note that the restriction of \eqref{def-form} to the body part corresponds to the matrix-valued Hermitian form associated to the Hardy space of $\mathbb C^{p\times q}$-valued functions analytic in the open unit disk.

\begin{corollary}
\label{contrac}
Let $S$ is be a rational function. Thus $S$ is a Schur function if and only if
\begin{equation}
\left[M_S F, M_S F\right] \preceq \left[F, F\right],
\label{schur-contractive-form}
\end{equation}
where $M_S$ is the operator defined by
\[
M_SF\equiv S\star F.
\]
\end{corollary}


\begin{corollary}
Let $S$ be a $(\overline{\Lambda}^{(1)})^{p\times q}$-valued rational function given by
\[
S(z)=s_0+zs_1+\cdots,
\]
where $s_0,s_1,\ldots\in (\overline{\Lambda}^{(1)})^{p\times q}$. Moreover, let $L_N$ denote the lower triangular block Toepliz matrix
\[
L_N=\begin{pmatrix}s_0    &0       &\cdots &0      &0\\
                   s_1    &s_0     &\cdots &0      &0\\
                   \vdots &\vdots  &\ddots &\vdots &\vdots\\
                   s_N    &s_{N-1} &\cdots &s_1    &s_0\end{pmatrix}.
\]
Then $S$ is a Schur-Grassmann function if and only if $L_N^*L_N\preceq I$ for every $N\in\mathbb{N}$.
\end{corollary}

\section{Wiener-Grassmann algebra}
\setcounter{equation}{0}
\label{sec-wiener}

Now, we introduce the Wiener algebra associated to the Grassmann algebra $\overline{\Lambda}^{(1)}$, or simply the Wiener-Grassmann algebra, which is defined by
\[
\mathcal{W}_G^p = \left\{f(z)=\sum_{n\in\mathbb{Z}}z^n f_n \suchthat z\in\overline{\Lambda}^{(1)}, f_n\in\left(\overline{\Lambda}^{(1)}\right)^{p\times p}, \Vert\left[f,f\right]\Vert_1<\infty\right\},
\]
where, in a way similar to \eqref{def-form}, $\left[\cdot,\cdot\right]$ is the Hermitian form defined by
\begin{equation}
\left[f,g\right] = \sum_{n\in\mathbb{Z}} g_n^* f_n,
\label{form-wiener}
\end{equation}
endowed with the star product introduced in \eqref{star-product}, i.e.,
\begin{equation}
f\star g(z) \equiv \sum_{n\in\mathbb{Z}} z^n \left(\sum_{u\in\mathbb{Z}} f_u g_{n-u}\right),
\label{star-product-wiener}
\end{equation}
where $f,g\in\mathcal{W}_G^p$.

\begin{remark}
{\rm The coefficients $f_n$ of $f$ could also take values in any closure $\left(\overline{\Lambda}^{(m)}\right)^{p\times p}$ with respect to a $m$-norm
\[
\Vert f_n\Vert_m = \left(\sum_{\alpha\in\mathfrak{I}_0}\left|(f_n)_\alpha\right|^m\right)^{1/m}.
\]
The elements such that $\Vert[f,f]\Vert_1<\infty$ would still be associated with a Wiener-Grassmann algebra.}
\end{remark}

An import subalgebra of $\mathcal{W}_G^p$ for our discussions here is the set $\mathcal{W}_{BP}^p$, which will be called the Wiener-Bochner-Phillips algebra and is defined by
\[
\mathcal{W}_{BP}^p = \left\{f(t)=\sum_{n\in\mathbb{Z}}e^{int} f_n \suchthat t\in\mathbb{R}, f_n\in\left(\overline{\Lambda}^{(1)}\right)^{p\times p}, \sum_{n\in\mathbb{Z}}\Vert f_n\Vert_1<\infty\right\},
\]
where $i\in\mathbb{C}$ is the imaginary unit. Note that the star product is reduced to the regular product in $\mathcal{W}_{BP}^p$.

\begin{lemma}
Let $f\in\mathcal{W}_{BP}^p$. Then, $f$ has an inverse in $\mathcal W_{BP}^p$ if and only if $f(t)$ has an inverse in $\overline{\Lambda}^{(1)}$, for every $t$.
\label{BP-lemma}
\end{lemma}

\begin{proof}
It is trivial that $f(t)$ is invertible for every $t$ if $f$ is invertible in $\mathcal{W}_{BP}^p$. The converse is just an adapted version of the original result presented by Bochner and Phillips in \cite[Theorem 1]{MR4:218g} to the case where the coefficients are elements of $\overline{\Lambda}^{(1)}$, instead of a generic non-commutative ring.
\end{proof}

The next result is the analogous of the Wiener-L\'{e}vy theorem in the present setting.

\begin{theorem}
Let $f\in\mathcal{W}_G^p$ and $f_{BP}\equiv f(e^{it})\in\mathcal{W}_{BP}^p$, where $t\in\mathbb{R}$. Then, the following statements are equivalent:
\begin{itemize}
\item[(i)]   $f$ is invertible in $\mathcal{W}_G^p$;
\item[(ii)]  $f_{BP}$ is invertible in $\mathcal W_{BP}^p$;
\item[(iii)] the body $\left(f_{BP}\right)_B$ of $f_{BP}$ is invertible in the classical Wiener algebra $\mathcal{W}^p$;
\item[(iv)]  $\left(f_{BP}\right)_B(t)\neq 0$ for every $t$.
\end{itemize}
\label{wiener-levy}
\end{theorem}

\begin{proof}
We divide the proof into a number of steps.\\

STEP 1: {\sl (i) is equivalent to (ii).}\smallskip

Observe that if (i) holds, there exists a function $g\in\mathcal{W}_G^p$ such that $f\star g(z)=1$ for every $z$. In particular, this is valid for $z=e^{it}$ for every $t\in\mathbb{R}$, i.e., $f_{BP}g_{BP}(t)=1$ for every $t$, where $g_{BP}\equiv g(e^{it})$. Conversely, if (ii) is true, there exists
\[
g_{BP} = \sum_{n\in\mathbb{Z}} e^{int} g_n \in\mathcal{W}_{BP}^p
\]
such that $f_{BP}g_{BP}=1$. Moreover, we have
\[
\sum_{u\in\mathbb{Z}} f_u g_{n-u} =
\left\{
\begin{array}{l}
0, \qquad n\neq0 \\
1, \qquad n=0
\end{array}
\right..
\]
Then, consider the natural ``extension'' of $g_{BP}$ to $\mathcal{W}_G$
\[
g = \sum_{n\in\mathbb{Z}} z^n g_n,
\]
where $z\in\overline{\Lambda}^{(1)}$, $z_B\neq 0$. Clearly, $f\star g=1$.\\

STEP 2: {\sl (ii) is equivalent to (iii).}\smallskip

To see that (ii) implies (iii), just note that $f_{BP}g_{BP}=1\Rightarrow \left(f_{BP}\right)_B\left(g_{BP}\right)_B=1$ and $\left(f_{BP}\right)_B, \left(g_{BP}\right)_B\in\mathcal{W}^p$. For the converse, we assume (iii) holds, which implies that $f_{BP}$ has an inverse in $\overline{\Lambda}^{(1)}$. By Lemma \ref{BP-lemma}, we conclude that $f_{BP}$ is invertible in $\mathcal{W}_{BP}^p$.\\

STEP 3: {\sl (iii) is equivalent to (iv).}\smallskip

This is just the classical Wiener-L\'{e}vy theorem.
\end{proof}

Similarly to the classical case, we also define
\[
{\mathcal{W}_G^p}_+ = \left\{ f\in\mathcal{W}_G^p \suchthat f_n = 0, n<0 \right\}
\]
and
\[
{\mathcal{W}_G^p}_- = \left\{ f\in\mathcal{W}_G^p \suchthat f_n = 0, n>0 \right\}.
\]

\begin{remark}
The Wiener algebra ${\mathcal{W}_G^p}_+$ coincides with $\Gamma^p$.
\end{remark}

A weak condition for invertibility of a function in ${\mathcal{W}_G^p}_+$ is presented in the following lemma.

\begin{lemma}
Let $f\in{\mathcal{W}_G^p}_+$. Thus $f$ has an inverse in ${\mathcal{W}_G^p}_+$ if and only if $\left[f(z)\right]_B \neq 0$ for every $z$ such that $\Vert z\Vert_1 \leq 1$.
\end{lemma}

\begin{proof}
Assume $f$ is invertible in ${\mathcal{W}_G^p}_+$. Then, there exists $g\in{\mathcal{W}_G^p}_+$ such that
\[
\left(\sum_{m=0}^\infty z^m f_m\right) \star \left(\sum_{n=0}^\infty z^n g_n\right) = 1
\]
fore every $z\in\overline{\Lambda}^{(1)}$. In particular, restricting the above expression to its body,
\[
\left(\sum_{m=0}^\infty z^m_B \left(f_m\right)_B\right) \left(\sum_{n=0}^\infty z^n_B \left(g_n\right)_B\right) = 1,
\]
i.e., $\left[f(z)\right]_B \neq 0$.

Conversely, assume $f\in{\mathcal{W}_G^p}_+$ and $\left[f(z)\right]_B \neq 0$ for every $z$ such that $\Vert z\Vert_1 \leq 1$. In particular, for $z=\xi\in\mathbb{C}$ with $|\xi|\leq 1$, $f_B(\xi)$ is an invertible element of $\mathcal{W}_+^p$. Hence, by Theorem \ref{wiener-levy}, $f$ is invertible in $\mathcal{W}_G^p$. Let $g\in\mathcal{W}_G^p$ be its inverse. Therefore,
\[
\left(\sum_{m=0}^\infty z^m f_m\right) \star \left(\sum_{n\in\mathbb{Z}} z^n g_n\right) = 1.
\]
Moreover, using again the fact that $f_B(\xi)$ is invertible in $\mathcal{W}_+^p$, we conclude that $\left(g_n\right)_B=0$ for every $n<0$.
\end{proof}

An analogous result to the above lemma holds for invertibility in ${\mathcal{W}_G^p}_-$. A strong version of the lemma, if it exists, should be the direct analogous of the classical result, i.e., it should characterize functions in ${\mathcal{W}_G^p}_+$ that have inverse in ${\mathcal{W}_G^p}_+$.

\section{Reproducing kernel Banach modules and interpolation}
\setcounter{equation}{0}
\label{sec-rkpm}

A space -- or module -- is said to admit a reproducing kernel if there exists a positive definite function $K(z,w)$ and a form $\left[\cdot,\cdot\right]$ such that every function in such a set can be pointwise evaluated as
\begin{equation}
f(z) = \left[f(\cdot),K(\cdot,z)\right].
\label{complex-ps}
\end{equation}
In the classical case, the form for a space of power series $f=\sum_{n=0}^\infty z^n a_n$, with $a_n\in\mathbb{C}$ and $\sum_{n=0}^\infty |a_n|^2<\infty$, is the a map into $\mathbb{C}$ and coincides with the usual inner product
\begin{equation}
\left[f,g\right]=\sum_{n=0}^\infty \overline{b}_n a_n,
\label{complex-form}
\end{equation}
where $g=\sum_{n=0}^\infty z^n b_n$, with $b_n\in\mathbb{C}$. Note that with respect to this product such a space of power series is a complex Banach algebra. Moreover, a similar definition could be given for matrix-valued coefficient $a_n$ of $f$.\smallskip

If one wants to replace the complex coefficients in \eqref{complex-ps} by elements in the Grassmannian setting, then it becomes clear that one type of power series of interest is the one defined by the Wiener algebra ${\mathcal{W}_G^p}_+$ studied in Section \ref{sec-wiener}.\smallskip

Since a form of the type \eqref{complex-form} is necessary, besides the star product, we also endow this set with the form defined in \eqref{def-form}. We define the notion of orthogonality in ${\mathcal{W}_G^p}_+$ according to such a form, i.e., we say $f\in{\mathcal{W}_G^p}_+$ is perpendicular to $g\in{\mathcal{W}_G^p}_+$ if $[f,g]=0$.\smallskip

Observe that if
\begin{equation}
K(z,w)=\sum_{n=0}^\infty z^n(w^\dag)^n,
\label{kernel}
\end{equation}
then the point evaluation of a function $f\in{\mathcal{W}_G}_+$ is
\[
f(z)=\left[f(\cdot),K(\cdot,z)\right].
\]

In this section, we are interested in solving an interpolation problem. In order to find all solutions of such a problem, we will need to take into consideration a particular reproducing kernel subset of rational functions of ${\mathcal{W}_G^p}_+$. For this, we need first to introduce the notion of linear independence. Since ${\mathcal{W}_G^p}_+$ is a module -- and not a linear space -- such a notion is delicate. Because of it, we present the following proposition.

\begin{proposition}
Consider $G=\{f_1,\hdots,f_N\}\subset\mathcal W^p_{G_+}$ such that $(f_1)_B,\hdots,(f_N)_B$ are linearly independent. Thus $G$ is a basis for $\text{span}(G)$.
\end{proposition}

\begin{proof}
We only need to prove that if
\begin{equation}
f_1 c_1 + \hdots + f_N c_N = 0,
\label{linear-combination}
\end{equation}
where $c_j\in\overline{\Lambda}^{(1)}$, $j=1,\hdots,N$, then $c_1=c_2=\hdots=c_N=0$. We show it by contradiction. Thus assume $f_1,\hdots,f_N$ are such that their body are linearly independent, and equation \eqref{linear-combination} is satisfied by coefficients $c_j\neq 0$. We start observing that the restriction of equation \eqref{linear-combination} to its body implies that the body of every $c_j$ must be null. Moreover, if we write
\[
c_j=\sum_{\alpha\in\mathfrak{I}} {c_j}_\alpha i_\alpha,
\]
let $m_j$ be the minimum value of $|\alpha|$ for which ${c_j}_\alpha\neq0$. Also, let $m=\min\{m_1,m_2,\hdots,m_N\}$. Hence, for every $\alpha$ such that $|\alpha|=m$, equation \eqref{linear-combination} implies that
\[
f_1 {c_1}_\alpha + \hdots + f_N {c_N}_\alpha = 0.
\]
Since at least one ${c_j}_\alpha$ is not null, it contradicts the fact that $f_1,\hdots,f_N$ are linearly independent. Therefore, equation \eqref{linear-combination} can only be satisfied if $c_1=c_2=\hdots=c_N=0$.
\end{proof}

Now, introduce the function $\Theta\in{\mathcal{W}_G^p}_+$, which is a fundamental object in our study and is given by
\begin{equation}
\Theta(z)=I_p - (1-z)C \star (I_q-z A)^{-\star} P^{-1} (I_q-A)^{-*} C^*J,
\label{theta}
\end{equation}
where $(C,A)\in\left(\overline{\Lambda}^{(1)}\right)^{p\times q}\times\left(\overline{\Lambda}^{(1)}\right)^{q\times q}$ is an observable pair, $J\in\left(\overline{\Lambda}^{(1)}\right)^{p\times p}$ is a signature matrix, and $P\in\left(\overline{\Lambda}^{(1)}\right)^{q\times q}$ is an invertible self-adjoint matrix.

\begin{remark}{\rm
A similar definition of $\Gamma^{p\times q}$ can be given for functions $F$ with the powers of $z$ on the right-hand side of the coefficients $f_n$. We call this space $\overline{\Gamma}^{p\times q}$. In such a space, the star product has to be replaced by the corresponding product
\begin{equation}
\left(\sum_{n=0}^\infty f_nz^n\right)\star_r \left(\sum_{n=0}^\infty g_nz^n\right) \equiv \sum_{n=0}^\infty\left(\sum_{u=0}^n f_ug_{n-u}\right)z^n.
\label{right-star-product}
\end{equation}
}
\end{remark}

\begin{proposition}
With the aforedefined $\Theta$, the equation
\[
\sum_{n=0}^\infty z^n\left(J-\Theta(z) J\Theta(w)^*\right)\left(w^\dag\right)^n = C\star(I_q-zA)^{-\star}P^{-1}\left[(I_q-wA)^*\right]^{-\star_r}\star_r C^*
\]
holds if and only if
\[
P-A^*PA=C^*JC.
\]
\label{rk-formula}
\end{proposition}

\begin{proof}
We first consider $z=\lambda$ and $w=\omega$ in $\mathbb{C}$. Writing $\alpha(\lambda)=C(I_m-\lambda A)^{-1}P^{-1}(I_m-A)^{-*}$ and $\beta=(I_q-A)^*P(I_q-A)$, observe that
\begin{eqnarray*}
J-\Theta(\lambda) J\Theta(\omega)^*\hspace{-2mm} & = &\hspace{-2mm} J-\left[I_p - (1-\lambda)C(I_q-\lambda A)^{-1} P^{-1} (I_q-A)^{-*} C^*J\right]J \\
     &&\hspace{15mm} \left[I_p - (1-\overline{\omega})JC(I_q-A)^{-1} P^{-1}\left[(I_q-\omega A)^*\right]^{-1} C^*\right] \\
     & = &\hspace{-2mm} \alpha(\lambda)\left[\beta-\lambda\beta\overline{\omega}+(1-\lambda)(1-\overline{\omega})(P-A^*PA-C^*JC)\right]\alpha(\omega)^*,
\end{eqnarray*}
which proves the proposition for $z=\lambda$. For an arbitrary $z\in\overline{\Lambda}^{(1)}$, the above calculation follows in a similar way.
\end{proof}
We, then, assume hereby that $P-A^*PA=C^*JC$. \smallskip

With the above definitions, we now want to solve the following interpolation problem.

\begin{problem}
Given an observable pair $(C,A)\in\left(\overline{\Lambda}^{(1)}\right)^{p\times q}\times \left(\overline{\Lambda}^{(1)}\right)^{q\times q}$ such that $C^*C=P-A^*PA$, we want to find every $f$ that satisfies
\begin{equation}
(C^*\star F)(A^*)=X.
\label{interpolation-problem}
\end{equation}
\end{problem}

First, we prove the following proposition, which presents a particular solution of the Problem \ref{interpolation-problem}.

\begin{proposition}
The function
\[
F_{min} = C\star(I_q-zA)^{-\star}P^{-1}X=\sum_{n=0}^\infty z^n CA^nP^{-1}X
\]
solves equation \eqref{interpolation-problem}.
\label{particular-solution}
\end{proposition}

\begin{proof}
This follows from a simple computation. In fact,
\[
(C^*\star F_{min})(A^*)=\sum_{n=0}^\infty (A^*)^n C^*CA^nP^{-1}X=PP^{-1}X=X.
\]
\end{proof}

With a particular solution, we now start looking for a general solution. Note that if $F$ is a different solution, then $G=F-F_{min}$ satisfies the homogeneous problem, i.e.,
\begin{equation}
(C^*\star G)(A^*)=0.
\label{homogeneous}
\end{equation}

In view of that, we present our next result.

\begin{proposition}
A function $G\in\mathcal W_{G_+}^p$ satisfies \eqref{homogeneous} if and only if
\begin{equation}
\left[G,C\star(I-zA)^{-\star}\xi\right]=0,
\label{orthogonality}
\end{equation}
for every $\xi$ in ${\mathcal{W}_G}_+^p$.
\label{prop-orthog}
\end{proposition}

\begin{proof}
Writing $G=\sum_{n=0}^\infty z^n g_n$, we obtain
\[
\sum_{n=0}^\infty (A^*)^n C^* g_n = 0.
\]
As a consequence, equation \eqref{orthogonality} is satisfied.
\end{proof}

Now, let $\mathcal{H}(\Theta)$ be the set of functions $f$ of the type
\[
f=C\star(I-zA)^{-\star}\xi.
\]
Note that such a space can be associated to the reproducing kernel
\[
K_{\mathcal{H}(\Theta)}(z,w)=C\star(I_q-zA)^{-\star}P^{-1}\left[(I_q-wA)^*\right]^{-\star_r}\star_r C^*,
\]
with $P-A^*PA=C^*C$. Using Proposition \ref{rk-formula}, we can rewrite it as
\begin{equation}
K_{\mathcal{H}(\Theta)}(z,w)=\sum_{n=0}^\infty z^n\left[I-\Theta(z)\Theta(w)^*\right](w^\dag)^n\xi.
\label{h-theta-kernel}
\end{equation}

\begin{proposition}
\label{propotheta}
${\mathcal{W}_G^p}_+$ can be decomposed as the direct sum
\begin{equation}
{\mathcal{W}_G^p}_+ = \Theta {\mathcal{W}_G^p}_+ \oplus \mathcal{H}(\Theta).
\label{wgplus-decomp}
\end{equation}
\end{proposition}

\begin{proof}
Using equation \eqref{h-theta-kernel}, note that the kernel $K$ defined in \eqref{kernel} satisfies
\[
K(z,w)\xi=\sum_{n=0}^\infty z^n I(w^\dag)^n\xi = \sum_{n=0}^\infty z^n\left[\Theta(z)\Theta(w)^*\right](w^\dag)^n\xi + K_{\mathcal{H}(\Theta)}(z,w)\xi
\]
for every $\xi\in\left(\overline{\Lambda}^{(1)}\right)^{p\times1}$. The two terms on the right-hand side are orthogonal to each other and, moreover, the term $\sum_{n=0}^\infty z^n\left[\Theta(z)\Theta(w)^*\right](w^\dag)^n\xi$ belongs to $\Theta {\mathcal{W}_G^p}_+$ and the term $K_{\mathcal{H}(\Theta)}(z,w)\xi$ belongs to $\mathcal{H}(\Theta)$. By Proposition \ref{prop-orthog}, the two sets are orthogonal to each other. Therefore, we conclude that \eqref{wgplus-decomp} holds.
\end{proof}

\begin{proposition}
All solutions of equation \eqref{interpolation-problem} are of the form
\[
f=f_{min}+\Theta\star h,
\]
where $h$ is an arbitrary element of ${\mathcal{W}_G}_+$.
\end{proposition}

\begin{proof}
This is a direct consequence of Propositions \ref{particular-solution}, \ref{orthogonality}, and \ref{propotheta}.
\end{proof}

\section{Nevanlinna-Pick interpolation}
\setcounter{equation}{0}
\label{sec-interpolation}

In this section, we study the Nevanlinna-Pick interpolation in $\overline{\Lambda}^{(1)}$. The classical complex version of this problem has already been discussed in the Introduction. Then, we go straight to the definition of the analogous problem we want to solve.

\begin{problem}
Given $N$ points $z_k$ in the open unit superdisk and $N$ values $s_k$ in $\overline{\Lambda}^{(1)}$, we want to characterize all Schur-Grassmann functions $S$ satisfying
\[
S(z_k)=s_k
\]
and such that the Pick matrix $P$ is superpositive, i.e.,
\[
P=\left(p_{jk}\right)\equiv\left(p_k(z_j;s_j)\right)\succeq0,
\]
where $p_k(z;s)$ is defined as
\[
p_k(z;s)\equiv(1-s s_k^\dag)\star(1-z z_k^\dag)^{-\star}
\]
for every $k\in\{1,\hdots,N\}$.
\label{np-interp}
\end{problem}

To solve the problem, and following the classical case (see e.g. \cite{bgr,Dym_CBMS}) 
we first consider the function $\Theta$ defined in \eqref{theta} with
\begin{equation}
\label{values}
A\equiv\left(
\begin{array}{c c c}
z_1^\dag &        &          \\
         & \ddots &          \\
         &        & z_N^\dag
\end{array}
\right), \
C\equiv\left(
\begin{array}{c c c}
1        & \cdots & 1        \\
s_1^\dag & \cdots & s_N^\dag
\end{array}
\right), \ \text{and} \
J\equiv\left(
\begin{array}{c c}
1 & 0  \\
0 & -1
\end{array}
\right).
\end{equation}
Note that $P-A^*PA=C^*JC$.

\begin{theorem}
With the above definitions,
\begin{equation}
\left(\begin{array}{c c}
1 & -s_k
\end{array}\right)
\star\Theta(z)\bigg\rvert_{z=z_k}=0
\label{interpo}
\end{equation}
for every $k\in\{1,\hdots,N\}$.
\end{theorem}

\begin{proof}
First, note that
\[
\left(\begin{array}{c c}
1 & -s_k
\end{array}\right)\star\Theta(z)\bigg\rvert_{z=z_k}=
\left(\begin{array}{c c}
1 & -s_k
\end{array}\right) -
(1-z)
\left(\begin{array}{c c}
1 & -s_k
\end{array}\right)C \star(I_N-zA)^{-\star}P^{-1}(I_N-A)^{-*}C^*J\bigg\rvert_{z=z_k}.
\]
Also,
\[
\left(\begin{array}{c c}
1 & -s_k
\end{array}\right)C =
\left(
\begin{array}{c c c}
1-s_ks_1^\dag & \cdots & 1-s_ks_N^\dag
\end{array}
\right)
\]
and, then,
\[
\left(\begin{array}{c c}
1 & -s_k
\end{array}\right)C \star (I_N-zA)^{-\star}\bigg\rvert_{z=z_k} = P^{[k]},
\]
where $P^{[k]}$ denotes the $k$-th row of $P$. Hence, if $I_N^{[k]}$ denotes the $k$-th row of $I_N$,
\begin{eqnarray*}
\left(\begin{array}{c c}
1 & -s_k
\end{array}\right)\star\Theta(z)\bigg\rvert_{z=z_k} & = &
\left(\begin{array}{c c}
1 & -s_k
\end{array}\right) -
(1-z_k) I_N^{[k]}(I_N-A)^{-*}C^*J \\
             & = & \left(\begin{array}{c c}
                            1 & -s_k
                         \end{array}\right) -
                   (1-z_k)(1-z_k)^{-1} I_N^{[k]}C^*J \\
             & = & 0,
\end{eqnarray*}
since $I_N^{[k]}C^*J = \left(\begin{array}{c c} 1 & -s_k \end{array}\right)$.
\end{proof}

With this result in mind, let $\sigma$ be any Schur-Grassmann function and consider the term
\[
\left(\begin{array}{c c}
1 & -S(z)
\end{array}\right)\star\Theta(z)\star
\left(
\begin{array}{c}
\sigma(z) \\
1
\end{array}
\right).
\]
Note that such a product is null at $z=z_k$ and, moreover,
\begin{equation}
s_k=\left(a\star\sigma(z)+b(z)\right)\star\left(c\star\sigma(z)+d(z)\right)^{-\star}\bigg\rvert_{z=z_k},
\label{sk-expression}
\end{equation}
with
\[
\Theta(z)=\left(\begin{array}{c c}
a(z) & b(z) \\
c(z) & d(z)
\end{array}\right)
\]
and provided that $c\star\sigma(z_k)+d(z_k)$ is invertible. In fact, that is the case since we only need the body of it to be invertible -- and the body correspond to the classical complex case, which is invertible.\smallskip

Equation \eqref{sk-expression} allows us to write
\begin{equation}
S\equiv T_\Theta(\sigma).
\label{S-operator}
\end{equation}
We need, then, to prove the following result.

\begin{proposition}
If $\sigma$ in \eqref{S-operator} is a Schur-Grassmann function, then $S$ is a Schur-Grassmann function.
\end{proposition}

\begin{proof}
Since
\[
\Theta^* J\Theta \preceq J,
\]
we have
\begin{eqnarray*}
\left(\begin{array}{c c} \sigma^\dag & 1 \end{array}\right)\star\Theta^* J\Theta\star\left(\begin{array}{c} \sigma \\ 1 \end{array}\right) \preceq \sigma^\dag\sigma-1 & \Rightarrow & \left(\begin{array}{c c} (a\star\sigma+b)^\dag  & -(c\star\sigma+d)^\dag \end{array} \right)\left(\begin{array}{c} a\star\sigma+b \\ c\star\sigma+d \end{array} \right) \preceq 0 \\
     & \Rightarrow & (a\star\sigma+b)^\dag(a\star\sigma+b)  \preceq (c\star\sigma+d)^\dag(c\star\sigma+d) \\
     & \Rightarrow & \left[(a\star\sigma+b)\star(c\star\sigma+d)^{-\star}\right]^\dag\left[(a\star\sigma+b)\star(c\star\sigma+d)^{-\star}\right] \preceq 1 \\
     & \Rightarrow & S^\dag S\preceq 1.
\end{eqnarray*}
\end{proof}

With the above result, we showed, then, how to obtain Schur-Grassmann functions that are solutions of the Problem \ref{np-interp}. With the next proposition, we want to show that those are all the solutions. 

\begin{proposition}
If a Schur-Grassmann function $S$ is a solution of the Problem \ref{np-interp}, then there exists a Schur-Grassmann function $\sigma$ such that $S$ is given by \eqref{S-operator}.
\end{proposition}

\begin{proof}
First, we let $\sigma$ be given by
\[
\sigma=T^{-1}_{\Theta}(S).
\]
With that, $S$ satisfies \eqref{S-operator}. Because the restriction to the body corresponds to a classical Nevanlinna-Pick interpolation problem in the
complex setting (with Pick matrix $P_B>0$), we know that $\sigma_B$ is also a Schur function. 
Therefore, by Proposition \ref{propo_equiv}, we conclude that $\sigma$ is a Schur-Grassmann function.
\end{proof}

\section{The Schur algorithm}
\setcounter{equation}{0}
\label{schuralgo}

Before discussing the counterpart of the Schur algorithm in the Grassmann setting we 
first go back to the recursion \eqref{recurschur}. Let for $\rho\in\mathbb D$,
\[
\Theta_{\rho}(\lambda)=\frac{1}{\sqrt{1-|\rho|^2}}\begin{pmatrix}1&\rho\\
\overline{\rho}&1\end{pmatrix}\begin{pmatrix}\lambda&0\\ 0&1\end{pmatrix}.
\]
Then, expressing $s_n$ in terms of $s_{n+1}$ we see that
\eqref{recurschur} can be rewritten as
\begin{equation}
\label{recur2!!!}
s_n(\lambda)=T_{\Theta_{\rho_n}(\lambda)}(s_{n+1}(\lambda)),
\end{equation}
with $\rho_n=s_n(0)$. The matrix-function $\Theta_{\rho_n}(\lambda)$ is $J$-inner in the open unit disk, and this property remains when multiplying $\Theta_{\rho_n}(\lambda)$ by a $J$-unitary constant, say $X_n$, on the right. Then, \eqref{recur2!!!} can be rewritten as
\begin{equation}
s_n(\lambda)=T_{\Theta_{\rho_n}(\lambda)X_n}(T_{X_n^{-1}}(s_{n+1}(\lambda))).
\end{equation}
Since $X_n$ is $J$-unitary, the function $T_{X_n^{-1}}(s_{n+1}(\lambda))$ is still a Schur function. Such a freedom was used in \cite[\S 3]{ad3} to develop the Schur algorithm in the matrix-valued case -- see, in particular, formula (4.13) in that paper.\smallskip

A particular choice of $X_n$ leads to 
\begin{equation}
\label{Mn!!!}
\Theta_{\rho_n}(\lambda)X_n=I_2-(1-\lambda)\frac{\begin{pmatrix}1 \\ \overline{\rho_n}
\end{pmatrix}\begin{pmatrix}1 &-\rho_n\end{pmatrix}}{1-|\rho_n|^2}.
\end{equation}
Denoting this last $J$-inner function by $M_n$, we rewrite the Schur algorithm -- following \cite{ad3} -- in the modified form
\begin{equation}
\begin{split}
\sigma_0(\lambda)&=s(\lambda)\\
\sigma_{n+1}(\lambda)&=T_{M_n(\lambda)^{-1}}(\sigma_n(\lambda))
\end{split}
\label{recu3!!!}
\end{equation}
This recursion with the counterpart of \eqref{Mn!!!} will be our definition of the
Schur algorithm in the Grassmann setting. More precisely, using \eqref{theta}
with (see \eqref{values})
\[
A=0,\quad C=\begin{pmatrix}1 \\ \rho_n^\dag \end{pmatrix}\quad {\rm and}\quad J=\left(
\begin{array}{c c}
1 & 0  \\
0 & -1
\end{array}
\right),
\]
the Stein equation becomes
\[
P=1-\rho_n\rho_n^\dag.
\]
Moreover, $M_n$ is
\begin{equation}
M_n(z)=I-(1-z)\begin{pmatrix}1 \\ \rho_n^\dag\end{pmatrix}(1-\rho_n\rho_n^\dag)^{-1}
\begin{pmatrix}1 &-\rho_n\end{pmatrix}=\begin{pmatrix}a(z)&b(z)\\ c(z)&d(z)\end{pmatrix},
\end{equation}
with
\begin{equation}
\begin{split}
a(z)&=1-(1-z)(1-\rho_n\rho_n^\dag)^{-1},\\
b(z)&=(1-z)(1-\rho_n\rho_n^\dag)^{-1}\rho_n,\\
c(z)&=-(1-z)\rho_n^\dag(1-\rho_n\rho_n^\dag)^{-1},\\
d(z)&=1+(1-z)\rho_n^\dag(1-\rho_n\rho_n^\dag)^{-1}\rho_n.
\end{split}
\end{equation}

Note that $M_n(0)$ is not invertible (see \eqref{interpo}),
\[
\begin{split}
\det M_n(0)&=\det\left(I-\begin{pmatrix}1 \\ \rho_n^\dag\end{pmatrix}(1-\rho_n\rho_n^\dag)^{-1}
\begin{pmatrix}1 &-\rho_n\end{pmatrix}\right)\\
&=1-(1-\rho_n\rho_n^\dag)^{-1}
\begin{pmatrix}1 &-\rho_n\end{pmatrix}\begin{pmatrix}1 \\ 
\rho_n^\dag\end{pmatrix}\\
&=0,
\end{split}
\]
even though, in general, it is not trivial to define the determinant of a matrix with entries in the Grassmann algebra.\smallskip

The following theorem follows directly from Proposition \ref{propo_equiv}.

\begin{theorem} (the Schur algorithm)
Let $S$ be a Schur-Grassmann function. Then, the recursion
\[
\begin{split}
\sigma_0(z)&=S(z)\\
\sigma_{n+1}(z)&=T_{M_n(z)^{-1}}(\sigma_n(z)),\quad with \,\, \rho_n=\sigma_n(0)
\end{split}
\]
defines a family of Schur-Grassmann functions as long as $P_n\succ 0$.
\end{theorem}

\section{Blaschke factors and Brune sections}
\setcounter{equation}{0}
\label{sec-blaschke}

Now, let $p=q=1$ and $J=1$. For such a case we relabel $b_a(z)\equiv\Theta(z)$, i.e.,
\[
b_a(z) = 1-(1-z)c\star(1-z a)^{-\star}p^{-1}(1-a)^{-\dag}c^\dag.
\]
where $a,c,p\in\overline{\Lambda}^{(1)}$, with $p$ being a real supernumber. Moreover, we constrain our study to the cases where the previous proposition apply -- explicitly, we assume $p-a^\dag pa=c^\dag c$. The function $b_a$ is called the Blaschke factor.

\begin{proposition}
The Blaschke factor $b_a$ vanishes at $\omega=c^{-\dag}a^\dag c^\dag$. Moreover, it can be factorized as
\[
b_a(z)=(z-\omega) \left[1+(\omega-1)c^{-\dag}pc^{-1}\omega^\dag\right]\star(1-z\omega^\dag)^{-\star}cp^{-1}(1-a)^{-\dag}c^\dag.
\]
\end{proposition}

\begin{proof}
First, note that
\begin{eqnarray*}
b_a(z) & = & 1-(1-z) c\star(1-z a)^{-\star}p^{-1}(1-a)^{-\dag}c^\dag \\
          & = & 1-(1-z)\sum_{n=0}^\infty z^n c a^np^{-1}(1-a)^{-\dag}c^\dag.
\end{eqnarray*}
Since $\omega=c^{-\dag}a^\dag c^\dag$, note that $\omega^n=c^{-\dag}(a^\dag)^nc^\dag$. Then,
\[
b_a(\omega) = 1-c^{-\dag}(1-a)^\dag \sum_{n=0}^\infty (a^\dag)^nc^\dag c a^n p^{-1}(1-a)^{-\dag} c^\dag.
\]
Because $p-a^\dag pa=c^\dag c$ implies that $\sum_{n=0}^\infty (a^\dag)^nc^\dag c a^n=p$, we conclude that $b_a$ in fact vanishes at $\omega$.

To show the factorization for $b_a(z)$, we start by observing that $a=c^{-1}\omega^\dag c$ and writing
\begin{eqnarray*}
b_a(z) & = & 1-(1-z)\sum_{n=0}^\infty z^n c a^np^{-1}(1-a)^{-\dag}c^\dag \\
       & = & 1-(1-z)\sum_{n=0}^\infty z^n (\omega^\dag)^ncp^{-1}(1-a)^{-\dag}c^\dag.
\end{eqnarray*}
We also need to note the equalities
\begin{eqnarray*}
c^{-\dag}pc^{-1} & = & c^{-\dag}(c^\dag c+a^\dag pa)c^{-1} \\
                 & = & 1+ c^{-\dag}a^\dag pac^{-1} \\
                 & = & 1+\omega c^{-\dag} p c^{-1}\omega^\dag
\end{eqnarray*}
and
\[
\sum_{n=0}^\infty \omega^n (\omega^\dag)^n = c^{-\dag}pc^{-1}.
\]
Therefore,
\begin{eqnarray*}
b_a(z) & = & b_a(z) - b_a(\omega) \\
       & = & \left[(z-1)\sum_{n=0}^\infty z^n (\omega^\dag)^n-(\omega-1)\sum_{n=0}^\infty \omega^n (\omega^\dag)^n\right]cp^{-1}(1-a)^{-\dag}c^\dag \\
       & = & \left[(z-1)(1-z\omega^\dag)^{-\star}-(\omega-1)\sum_{n=0}^\infty \omega^n (\omega^\dag)^n\right]cp^{-1}(1-a)^{-\dag}c^\dag \\
       & = & \left[(z-1)-(\omega-1)c^{-\dag}pc^{-1}\star(1-z\omega^\dag)\right]\star(1-z\omega^\dag)^{-\star}cp^{-1}(1-a)^{-\dag}c^\dag \\
       & = & (z-\omega) \left[1+(\omega-1)c^{-\dag}pc^{-1}\omega^\dag\right]\star(1-z\omega^\dag)^{-\star}cp^{-1}(1-a)^{-\dag}c^\dag.
\end{eqnarray*}
\end{proof}

As consequence of Proposition \ref{propotheta} we can now state:

\begin{proposition}
Let $a\in\overline{\Lambda}^{(1)}$ with $aa^\dag=1$. Moreover, let $f\in\mathcal{W}_G$. Thus, $f(a^\dag)=0$ if and only if there exists $g\in\mathcal{W}_G$ such that $f(z)=b_a(c^{-\dag}zc^\dag)\star g(z)$. Moreover, $\Vert f\Vert_{\mathcal{W}_G}=\Vert g\Vert_{\mathcal{W}_G}$.
\end{proposition}

Now, let us return to the original definition of $\Theta$ and study the analogous of the Blaschke-Potapov factors of the third kind, also known as the Brune section (named after Brune; see \cite{brune-31}). In such a case, we consider
\begin{equation}
c^*Jc=0,
\end{equation}
which leads to
\begin{equation}
p=a^\dag pa,
\label{p=apa}
\end{equation}
where $p$ is a real supernumber and
\begin{equation}
a^\dag a=1.
\label{aa=1}
\end{equation}
Note that it forces this last condition also implies that $a$ commutes with $a^\dag$. In last instance, it implies that the odd terms of $a$ are real supernumbers. Furthermore, conditions \eqref{p=apa} and \eqref{aa=1} together reveal that $a$ and $p$ commute as well. Therefore, $p$ should not only be a real supernumber, it should also be even. \smallskip

Our goal here is to rewrite $\Theta$ in its most known decomposition form, which in the complex case is
\[
\Theta_{BP}(\lambda)=I+\frac{cc^*J}{2p}\frac{\lambda+a}{\lambda-a}.
\]
We note that
\[
\Theta(\lambda)\Theta_{BP}(\lambda)^{-1}=I-\frac{cc^*J}{2p}\frac{1+\overline{a}}{1-\overline{a}},
\]
which suggests a way to obtain an analogous to $\Theta_{BP}$ from $\Theta$ in the case of $\overline{\Lambda}^{(1)}$-valued variables. We define
\[
M=I-\frac{1}{2}c(1+a)^\dag p^{-1} (1-a)^{-\dag}c^*J.
\]
Then, a simple computation shows that
\[
\Theta_{BP}(z)=\Theta(z)M^{-1}=I+\frac{1}{2}(z+a^\dag)\star(z-a^\dag)^{-\star}c(1-a)^\dag p^{-1}(1-a)^{-\dag}c^*J.
\]

\section{Final remarks}
\setcounter{equation}{0}
\label{sec-final}
In this work we have introduced the notion of positivity in the Grassmann algebra setting. We also defined the (left) Cauchy product for power series -- as well as its adjoint, the right Cauchy product. Such ideas were used to study the counterpart of classical problems, such as the one step extension problem for Toeplitz matrices and the Wiener algebra, and to begin the development of Schur analysis.\smallskip

There are multiple research directions that emerge from our results. For instance, they could be used as a starting point for the construction of functional analysis tools -- and, in particular, the relevant theory of reproducing kernel Hilbert modules. The approach for stochastic processes in the Grassmann setting we introduced in \cite{2018arXiv180611058A} could be also enriched with more studies in this direction.\smallskip

Moreover, the classical problem of the one-step extension of Toeplitz matrices, as mentioned in Section \ref{sec-matrix-ext}, is also associated to stochastic processes and, in particular, to the Yule-Walker equations. Another topic of investigation is what is the connection between those topics in the Grassmann setting. In case such a connection exist, an interesting question is concerning the stochastic processes that arrise from the one-step extension. Are they part of the same type of processes we introduced in \cite{2018arXiv180611058A}?\smallskip

Finally, due to the importance of Grassmann numbers in quantum field theory and the physical motivation for the classical complex counterpart of what was studied here, one could investigate what are the applications in physics of this work and of future works motivated by it.

\section*{Acknowledgments}
Daniel Alpay thanks the Foster G. and Mary McGaw Professorship in Mathematical Sciences, which supported this research. Ismael L. Paiva acknowledges financial support from the Science without Borders program (CNPq/Brazil). Daniele C. Struppa thanks the Donald Bren Distinguished Chair in Mathematics, which supported this research.

\bibliographystyle{plain}
\bibliography{all}
\end{document}